\documentclass[12pt]{amsart}
\usepackage[utf8]{inputenc}
\usepackage[lite, alphabetic, nobysame]{amsrefs}
\usepackage{amsmath,mathtools}
\usepackage{amssymb}
\usepackage{srcltx}
\usepackage{xcolor}
\usepackage{tikz}
\usepackage{xspace}
\usepackage{enumerate}
\usepackage{geometry}
\usepackage{marginnote}

\newcommand{\set}[2]{ \left\{ #1 :\, #2 \right\} }
\newcommand{\seqq}[2]{ \left\langle #1 :\, #2\right\rangle }

\newcommand{\Fraisse}{Fra\"iss\'e\xspace}

\newtheorem{theorem}{Theorem}
\newtheorem{lemma}[theorem]{Lemma}
\newtheorem{corollary}[theorem]{Corollary}
\newtheorem{proposition}[theorem]{Proposition}
\newtheorem*{claim}{Claim}

\theoremstyle{definition}
\newtheorem{definition}[theorem]{Definition}
\newtheorem{remark}[theorem]{Remark}
\newtheorem{question}[theorem]{Question}
\newtheorem{fact}[theorem]{Fact}
\newtheorem{example}[theorem]{Example}

\title[Strong ergodicity  for  shifts of bounded algebraic dimension ]{
Strong  ergodicity phenomena for Bernoulli shifts of bounded algebraic dimension 
}

\author{Aristotelis Panagiotopoulos}
\address{University of Vienna, Institute for Mathematics, Kolingasse 14-16, 1090 Vienna, Austria}
\email{aristotelis.panagiotopoulos@gmail.com}



\author{Assaf Shani}
\address{Department of Mathematics and Statistics,
Concordia University, 1455 De Maisonneuve Blvd. W.
Montreal, QC  H3G 1M8
CANADA}
\email{assaf.shani@concordia.ca}

\begin{document}

\begin{abstract}
The algebraic dimension of a Polish permutation group $Q\leq \mathrm{Sym}(\mathbb{N})$ is the size of the largest $A\subseteq \mathbb{N}$ with the property that  the orbit of every $a\in A$ under the  pointwise stabilizer of $A\setminus\{a\}$ is infinite. 
We study the Bernoulli shift $P\curvearrowright \mathbb{R}^{\mathbb{N}}$ for various Polish permutation groups $P$ and we provide criteria under which the $P$-shift is generically ergodic relative to the injective part of the $Q$-shift, when $Q$ has algebraic dimension $\leq n$.
We use this to show that the sequence of pairwise  $*$-reduction-incomparable  equivalence relations defined in \cite{KP} is  a strictly increasing sequence in the  Borel reduction hierarchy. We also use our main theorem to exhibit an equivalence relation of pinned cardinal $\aleph_1^{+}$ which strongly resembles the equivalence relation of pinned cardinal $\aleph_1^{+}$ from \cite{Zap11}, but which does not Borel reduce to the latter. It remains open whether they are actually incomparable under Borel reductions.

Our proofs rely on the study of  symmetric models whose symmetries come from the group $Q$.  
We show that when $Q$ is ``locally finite"---e.g. when $Q=\mathrm{Aut}(\mathcal{M})$, where $\mathcal{M}$ is 
the \Fraisse{} limit of a \Fraisse{} class satisfying the disjoint amalgamation property---the corresponding symmetric model admits a theory of supports which is  analogous to that in the basic Cohen model.

%We also study the pinned cardinals of these equivalence relations.
%We illustrate how these techniques can provide  finer results than pinned cardinality methods. 
%For example, we exhibit an equivalence relation of pinned cardinal $\aleph_1^{+}$ which strongly resembles the one in \cite{Zap11}, but which does not Borel reduce to the later. It remains open whether they are actually incomparable under Borel reductions.
\end{abstract}

\maketitle

\section{Introduction}

Let $G$ be a Polish group acting continuously  $G\curvearrowright X$  on a Polish space $X$. A fundamental question which motivates several threads of research, both in  descriptive set theory and in ergodic theory, is how much information  does the associated orbit equivalence relation $E^G_X$ preserve from the original action. 
 In particular, one would like to know  which structural properties of the group $G$ can  produce dynamical phenomena which are strong and distinct enough, so as to have an impact on the ``complexity" of  $E^G_X$; the latter being measured by the  position of $E^G_X$ in the \emph{Borel reduction hierarchy} $\leq_B$, as well as by  its \emph{strong ergodic properties}.

When it comes to Polish permutation groups $P\leq \mathrm{Sym}(\mathbb{N})$, an action which often serves as bridge between structural and dynamical properties of $P$ is the {\bf Bernoulli shift}  $P\curvearrowright \mathbb{R}^{\mathbb{N}}$ of $P$. For example, if $E(P)$ is its orbit equivalence relation, then:
\begin{enumerate}
\item $E(P)$ is \emph{smooth}  if and only if $P$ is compact, see \cite{KMPZ};
\item $E(P)$ is \emph{essentially countable} if and only if $P$ is locally compact, \cite{KMPZ};
\item $E(P)$ is \emph{classifiable by CLI-group actions} if and only if  $P$ is CLI, \cite{KMPZ};
\item $E(P)\lneq_B E(P \wr \Gamma )$ for various countable groups $\Gamma$, \cite{CC20};
\item $E(\mathbb{Z} \wr  \mathbb{Z})$ is not \emph{classifiable by TSI-group actions}, \cite{AP20}.
\item $E(\Gamma \wr  \Gamma)\not\leq_B E(\Delta \wr  \Delta)$ for various countable $\Gamma,\Delta$, \cites{Sha19}.
\end{enumerate}

In fact, the above results still hold if we replace all shifts with their \emph{injective parts}.  By the injective part $E_{\mathrm{inj}}(P)$ of 
$E(P)$ we mean the restriction of $E(P)$ to the $P$-invariant subset $\mathrm{Inj}(\mathbb{N},\mathbb{R})$ of $\mathbb{R}^{\mathbb{N}}$, consisting   of all injective sequences.

In this paper we study how the  \emph{algebraic dimension}  of a Polish permutation group affects the ergodic behavior of its Bernoulli shift.  For every Polish permutation group  $P\leq \mathrm{Sym}(\mathbb{N})$ and every $A\subseteq \mathbb{N}$, we denote the pointwise stabilizer of $A$ by $P_A$. The {\bf algebraic closure} of $A$ with respect to $P$ is the set:
\[ [A]_P:=\{a \in \mathbb{N} \colon \text{the orbit } P_{F}\cdot a \text{ is finite for some finite } F\subseteq A\}.\]
The algebraic dimension $\mathrm{dim}(P)$ of the permutation group $P$ is simply the dimension of  the closure operator $A\mapsto [A]_P$. It corresponds to the size of its largest ``independent" set with respect to the closure operator $A\mapsto [A]_P$. More precisely:

\begin{definition}\label{Def1}
The {\bf algebraic dimension} $\mathrm{dim}(P)$ of  a permutation group $P$  is the largest $n\in\mathbb{N}\cup\{\infty\}$ for which there is $A\subseteq \mathbb{N}$, with  $|A|=n$, so that   $a\not\in [A\setminus \{a\}]_P$ for all $a\in A$. If $n\not=\infty$, we say that $P$ is of {\bf bounded algebraic dimension}.
\end{definition}

When $P:=\mathrm{Aut}(\mathcal{N})$ is the automorphism group of a countable structure $\mathcal{N}$, it is often then case that the dynamical definition of algebraic dimension above coincides with the model theoretic one; see \cite{TZ12}. In this context, algebraic dimension has  extensively been studied when  $A\mapsto [A]_{P}$ forms a pregeometry, i.e., when it additionally satisfies the \emph{exchange principle}. Examples include the case where $\mathcal{N}$ is the $n$-dimensional $\mathbb{Q}$-vector space or the complete $n$-branching tree.  It is easy to see that this additional assumption forces $P=\mathrm{Aut}(\mathcal{N})$  to be locally compact and, in turn, $E(P)$  to be essentially countable \cite{Kec92}. From this point of view,  bounded algebraic dimension can be seen as a partial weakening of  local compactness.

The following theorem shows that the strong ergodic properties of the Bernoulli shift of $P$, as well as the Borel reduction complexity of $E(P)$, are both sensitive to the algebraic dimension of $P$. We say that $P$ is {\bf locally finite} if the associated closure operator $A\mapsto [A]_P$ is locally finite, i.e., if $[A]_P$ is finite for all finite $A$. We say that   $P$ is {\bf $\boldsymbol{n}$-free} if for all finite $F\subseteq \mathbb{N}$ there are $g_0,\ldots,g_{n-1}\in P$, so that for all $i\in\{0,\ldots, n-1\}$, $[g_i\cdot F]_P$ and    $[ \bigcup_{j : j\neq i} (g_j\cdot F)]_{P}$ are disjoint. Notice that ``$n$-freedom" of   $P$ implies that its algebraic dimension is ``everywhere" at least $n$: appropriate translates of any given  $a\in \mathbb{N}$ can be used to provide an ``independent" set  $A:=\{ g_0\cdot a,\ldots,g_{n-1}\cdot a\}$ of size $n$, witnessing that $\mathrm{dim}(P)\geq n$.

\begin{theorem} \label{T:1} Let $P,Q\leq \mathrm{Sym}(\mathbb{N})$ be Polish permutation groups and let $n\in\omega$. If the algebraic dimension of  $Q$ is  at most $n$ and $P$ is locally finite and $(n+1)$-free,  then $E_{\mathrm{inj}}(P)$ is generically $E_{\mathrm{inj}}(Q)$-ergodic. 
As a consequence, we have that: 
\[ E_{\mathrm{inj}}(P)\not\leq_B E_{\mathrm{inj}}(Q).\]
\end{theorem}

While it is tempting  to conjecture that one should  be able to  replace  $E_{\mathrm{inj}}(Q)$ with the full shift $E(Q)$ above, one of three ``red herrings" from \cite{BFKL15} suggests that  Theorem \ref{T:1} could actually be sharp; see Remark \ref{R:F} below.

Baldwin, Koerwien and Laskowski defined  for each $n\geq 2$ a countable, locally finite, ultrahomogeneous structure $\mathcal{N}^{\mathrm{BKL}}_n$ whose $\mathcal{L}_{\omega_1\omega}$-theory has models of size $\aleph_{n-1}$ but not of size $\aleph_{n}$; see \cite{BKL17}.
The construction of the $\aleph_{n-1}$-sized model was based on the fact that the category $\mathrm{Age}(\mathcal{N}^{\mathrm{BKL}}_n)$, of all finite substructures of $\mathcal{N}^{\mathrm{BKL}}_n$, has \emph{$m$-disjoint amalgamation} for all $m\leq n$, and the lack of an $\aleph_{n}$-sized models was based on the fact that the algebraic dimension of $\mathcal{N}^{\mathrm{BKL}}_n$ is $n$.

Kruckman and Panagiotopoulos~\cites{KP} studied \emph{$*$-reductions}, that is, reductions which preserve comeager sets. It follows from their results that the collection:
\[E_{\mathrm{inj}}(\mathrm{Aut}(\mathcal{N}^{\mathrm{BKL}}_2)), E_{\mathrm{inj}}(\mathrm{Aut}(\mathcal{N}^{\mathrm{BKL}}_3)),\ldots, E_{\mathrm{inj}}(\mathrm{Aut}(\mathcal{N}^{\mathrm{BKL}}_n)), \ldots, E_{\mathrm{inj}}(\mathrm{Sym}(\mathbb{N}))\]
consists of pairwise incomparable equivalence relations under $*$-reductions. However, the question of how they compare with respect to $\leq_B$ remained  open. The fact that $\mathrm{Age}(\mathcal{N}^{\mathrm{BKL}}_{n+1})$  satisfies  \emph{disjoint $m$-amalgamation} for all $m\leq n+1$ can be used to show that  $\mathrm{Aut}(\mathcal{N}^{\mathrm{BKL}}_{n+1})$ satisfies the assumption for $P$ in Theorem \ref{T:1}. Hence we have:

\begin{corollary} \label{cor:1} Assume that $F=E_{\mathrm{inj}}(\mathrm{Aut}(\mathcal{N}^{\mathrm{BKL}}_m))$ and  either $E=E_{\mathrm{inj}}(\mathrm{Aut}(\mathcal{N}^{\mathrm{BKL}}_n))$ with $n>m$, or $E=E_{\mathrm{inj}}(\mathrm{Sym}(\mathbb{N}))$. Then, $E$ is generically $F$-ergodic. In particular, $E_{\mathrm{inj}}(\mathrm{Aut}(\mathcal{N}^{\mathrm{BKL}}_{n+1}))$ is not Borel reducible to $E_{\mathrm{inj}}(\mathrm{Aut}(\mathcal{N}^{\mathrm{BKL}}_n))$.

\end{corollary}
We do not know if $E_{\mathrm{inj}}(\mathrm{Aut}(\mathcal{N}^{\mathrm{BKL}}_n))\leq_B E_{\mathrm{inj}}(\mathrm{Aut}(\mathcal{N}^{\mathrm{BKL}}_{n+1}))$. In Section~\ref{subsec : positive reduction} we show that a different family of equivalence relations, closely related to $E_{\mathrm{inj}}(\mathrm{Aut}(\mathcal{N}^{\mathrm{BKL}}_n))$, is strictly increasing with respect to Borel reducibility. 
These equivalence relations were also shown  to be pairwise incompatible with respect to $*$-reductions in \cites{KP}.

In contrast to \cite{KP}, which uses a classical argument based on the amalgamation properties of $\mathcal{N}^{\mathrm{BKL}}_n$, our 
 arguments here employ set-theoretic symmetric model techniques, as developed in \cite{Sha18} and \cite{Sha19},  and they directly utilize the ``algebraicity" properties of $\mathcal{N}^{\mathrm{BKL}}_n$.  
 When it comes to the second statement of Corollary \ref{cor:1} (the Borel irreducibility conclusion), there is yet a third way to approach it. Namely, 
using a set-theoretic pinned cardinality argument, as in \cites{Zap11,LZ}, and the fact that the $\mathcal{L}_{\omega_1,\omega}$-theory of $\mathcal{N}^{\mathrm{BKL}}_n$ has models of size $\aleph_{n-1}$ but not of size $\aleph_{n}$; see Section \ref{sec;pinned}.
However, the analysis we provide here is sharper in a few ways.
First, while pinned cardinality methods suffice for proving Borel irreducibility, Theorem~\ref{T:1} establishes stronger generic ergodicity results. Second, using our methods we are able to distinguish between two equivalence relations which have the same pinned cardinals.
%Howeverthe analysis we provide here based on symmetric model techniques is sharper in a few ways.  For one, if the closure operator $A\mapsto [A]_Q$ associated to $Q$ in Theorem \ref{T:1} is not ``locally finite," then the pinned cardinality  of the Bernoulli shift of $Q$ cannot be bounded by $(\aleph_n)^{+}$. For two, even if a pinned cardinality argument is plausible, it only addresses the relation $\leq_B$ without sorting out the strong ergodic properties as in Theorem \ref{T:1}. 
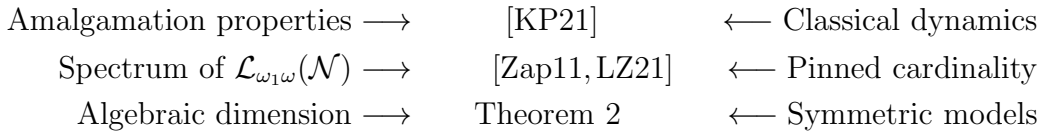
\begin{align*}
\text{Amalgamation properties} \longrightarrow &\quad\quad\quad  \text{\cite{KP}}& \longleftarrow \text{Classical dynamics}\\
\text{Spectrum of }\mathcal{L}_{\omega _{1} \omega }(\mathcal{N}) \longrightarrow & \quad \quad\;\; \text{\cites{Zap11,LZ}}& \longleftarrow  \text{Pinned cardinality}\\
\text{Algebraic dimension} \longrightarrow & \quad\quad\text{Theorem \ref{T:1}}& \longleftarrow \text{Symmetric models}
\end{align*}

Pinned equivalence relations form a bounded complexity class  within the Borel reduction hierarchy, which has a forcing-theoretic definition. The most central example of a non-pinned equivalence relation is $=^+$, which can be identified to $E_{\mathrm{inj}}(\mathrm{Sym}(\mathbb{N}))$. Answering a question of Kechris, Zapletal showed that $=^+$ is not the minimal unpinned Borel equivalence relation  \cite{Zap11}.  Essential in his proof was the notion of pinned cardinality, which is a measure of how far an equivalence relation is from being pinned. In particular, he provided a transfinite sequence of unpinned equivalence relations below  $=^+$ that is linearly ordered under  $\leq_B$.
The first $\omega$ examples in this sequence share many properties with the sequence of the $E_{\mathrm{inj}}(\mathrm{Aut}(\mathcal{N}^{\mathrm{BKL}}_n))$ in Corollary \ref{cor:1} above. 
The minimal example in Zapletal's sequence is defined as follows: 
first, \cites{Zap11} considers a certain Borel graph on $\mathbb{R}$, which has cliques of size $\aleph_1$ but no larger, and then defines $Z\subset \mathbb{R}^{\mathbb{N}}$ as the set of all injective sequences of reals which enumerate a clique in this graph.
The restricted equivalence relation $=^+\restriction Z$ is then unpinned and strictly below $=^+$ in Borel reducibility. The key point distinguishing $=^+\restriction Z$ is that its pinned cardinality is $\aleph_1^+$, while the pinned cardinality of $=^+$ is $\mathfrak{c}^+$.
The equivalence relation  $E_{\mathrm{inj}}(\mathrm{Aut}(\mathcal{N}^{\mathrm{BKL}}_2))$ has pinned cardinality $\aleph_1^{+}$ as well. Hence, arguments based on pinned cardinality cannot separate it from $=^+\restriction Z$.  However, as a consequence of Theorem~\ref{T:1} we prove the following.

\begin{theorem}\label{thm;BKL2-vs-F}
$E_{\mathrm{inj}}(\mathrm{Aut}(\mathcal{N}^{\mathrm{BKL}}_2))$ is generically $=^+\restriction Z$-ergodic and hence: 
\[E_{\mathrm{inj}}(\mathrm{Aut}(\mathcal{N}^{\mathrm{BKL}}_2)){\not\leq_B} =^+\restriction Z\]
\end{theorem}

\begin{question}
Is $=^+\restriction Z$ Borel reducible to $E_{\mathrm{inj}}(\mathrm{Aut}(\mathcal{N}^{\mathrm{BKL}}_2))$? If not, do they form a minimal basis for non-pinned equivalence relations?
\end{question}

\begin{question}\label{Q:6}
Can we replace  $E_{\mathrm{inj}}(Q)$ with the full shift $E(Q)$ in Theorem~\ref{T:1}?
\end{question}

To appreciate the subtleties in addressing Question \ref{Q:6}, notice that, despite its bounded algebraic dimension, $Q$ can potentially admit as complex actions as $\mathrm{Sym}(\mathbb{N})$:

\begin{remark}\label{R:F}
Let $Q:=\mathrm{Aut}(\mathcal{N}^{\mathrm{BKL}}_n)$, for some $n\geq 2$. While by Corollary \ref{cor:1} we have that $E_{\mathrm{inj}}(\mathrm{Sym}(\mathbb{N}))\not\leq_B E_{\mathrm{inj}}(Q)$, it turns out that any orbit equivalence relation that is induced by an action of $\mathrm{Sym}(\mathbb{N})$ can be Borel reduced to some action of $Q$. This follows from a classical result of Mackey \cite{Gao09}*{Theorem 3.5.2} and the fact that $Q$ contains a closed subgroup which admits a continuous homomorphism onto  $\mathrm{Sym}(\mathbb{N})$. The latter constitutes, essentially, one of the three ``red herrings" from \cites{BFKL15}.  Indeed,  $\mathcal{N}^{\mathrm{BKL}}_n$ is the \Fraisse{} limit of a \Fraisse{} class which satisfies the disjoint amalgamation property; see Example \ref{Ex:2}. Hence, by Theorem 1.10.(1) in \cite{BFKL15}, the domain $\mathbb{N}$ of the structure $\mathcal{N}^{\mathrm{BKL}}_n$ can be endowed with an equivalence relation $E$ which partitions it into countably many disjoint pieces so that: every permutation in  $\mathrm{Sym}(\mathbb{N}/E)$ lifts to some element of the subgroup $Q_E$ of  $Q$,  consisting of all $E$-preserving elements of $Q$, i.e., $g\in Q_E$ if and only if $a E b \iff ga E gb$. 
\end{remark}

\subsection*{Acknowledgments} We would like to thank  A. Kruckman and S. Allison  for  numerous enlightening and inspiring conversations.

\section{Definitions and preliminaries}\label{S:Prelim}

\subsection{Borel reductions and strong ergodicity}

Let $E$ and $F$ be analytic equivalence relation on the Polish spaces $X$ and $Y$, respectively. A  {\bf Borel reduction} from $E$ to $F$ is any Borel function $f\colon X\to Y$, with $x \mathrel{E} x' \iff f(x) \mathrel{F} f(x')$, for all $x,x'\in X$.  If there exists a Borel reduction from $E$ to $F$, we say that $E$ is {\bf Borel reducible} to $F$ and we denote this by $\leq_B$. We say that $E$ is {\bf smooth} if $E\leq_B \; =_{\mathbb{R}}$, where $=_{\mathbb{R}}$ is equality on $\mathbb{R}$.  We say that $E$ is {\bf essentially countable} if $E\leq_B F$, where $F$ is a Borel equivalence relation with countable $F$-classes.  We say that $E$ is {\bf classifiable by countable structures} if $E\leq_B \; \simeq_{\mathrm{iso}}$, where $\simeq_{\mathrm{iso}}$ is the isomorphism equivalence relation on the Polish space of all $\mathcal{L}$-structures on $\mathbb{N}$ for some first order language $\mathcal{L}$; see \cite{Kan08}*{12.3}, \cite{Gao09}, \cite{Hjo00}.

Often the non-existence of a Borel reduction from $E$ to $F$ is a consequence of a strong ergodicity phenomenon. 
A {\bf Baire-measurable homomorphism} from $E$ to $F$ is any Baire measurable function $f\colon X\to Y$, with $x E x' \implies f(x) F f(x')$, for all $x,x'\in X$. We say that  \textbf{$E$ is generically $F$-ergodic} if any Baire measurable homomorphism from $E$ to $F$ maps a comeager subset of $X$ into a single $F$-class. For example, $E$ is generically $=_\mathbb{R}$-ergodic if and only if \textbf{$E$ is generically ergodic}, that is, every $E$-invariant subset of $X$ is either meager or comeager.
Notice that if every  $E$-class is meager in $X$ and $E$ is generically $F$-ergodic, then $E\not\leq_B F$. All the examples considered in this paper have meager equivalence classes; see Section~\ref{SS:BS}.

If $G\curvearrowright X$ is a continuous action of a Polish group on a Polish space $X$, then the associated {\bf orbit equivalence relation} is the relation $E^G_X$ on $X$ defined by: 
\[x \mathrel{E^G_X} x' \iff \exists g \in G \; (gx=x').\]   
%It is well known that  $E^G_X$ is smooth whenever $G$ is compact. Moreover, by \cite{Kec92}, $E^G_X$ is essentially countable whenever $G$ is locally compact.

\subsection{Bernoulli shifts} \label{SS:BS}

Let $\mathrm{Sym}(\mathbb{N})$ be the Polish group of all permutations $g\colon \mathbb{N}\to \mathbb{N}$ of $\mathbb{N}$ endowed with the pointwise convergence topology. A {\bf Polish permutation group} $P$ is any closed subgroup of  $\mathrm{Sym}(\mathbb{N})$. The action of $P$ on $\mathbb{N}$, by $(g,a)\mapsto  g(a),$ induces an action of $P$ on the space $\mathbb{R}^{\mathbb{N}}$, of all maps $x\colon \mathbb{N}\to \mathbb{R}$,  by $(g,x)\mapsto g x$, where $(g x)(a)=x ( g^{-1}(a))$. Following \cite{KMPZ}, we call $P\curvearrowright \mathbb{R}^{\mathbb{N}}$ the {\bf Bernoulli shift of $P$} and we denote by $E(P)$ the associated orbit equivalence relation on $\mathbb{R}^{\mathbb{N}}$. The {\bf injective part of the Bernoulli shift} is the restriction of the above action to the $P$-invariant subset $\mathrm{Inj}(\mathbb{N},\mathbb{R})$ of $\mathbb{R}^{\mathbb{N}}$ which consists of all injective functions of $\mathbb{R}^{\mathbb{N}}$. Notice that $\mathrm{Inj}(\mathbb{N},\mathbb{R})$ is a dense $G_{\delta}$ subset of $\mathbb{R}^{\mathbb{N}}$. Moreover, it is easy to see that the orbit of any $x\in \mathbb{R}^{\mathbb{N}}$ under the action $P\curvearrowright \mathbb{R}^{\mathbb{N}}$ is meager. This follows  for example from the fact that, for all $k\in\mathbb{N}$, the projection $\pi_k\colon \mathbb{R}^{\mathbb{N}}\to \mathbb{R}$ to the $k$-th coordinate is continuous and open, and hence, $\pi_k^{-1}(r)$ is nowhere dense for all $r\in\mathbb{R}$.

\subsection{Automorphism groups and ultrahomogeneous structures}\label{SS:BS'}

Several connections between the Borel reduction complexity of $E(P)$ and the dynamical properties of $P$ are summarized in the introduction. Often these dynamical properties of $P$ are reflections  of combinatorial and model-theoretic properties of first-order structures. The connection with model-theory comes from the fact that every Polish permutation group $P$ can be expressed as the automorphism group $\mathrm{Aut}(\mathcal{N})$ of some countable structure $\mathcal{N}$. The connection with combinatorics comes from the fact that this countable structure $\mathcal{N}$ can in turn always be assumed to be \emph{ultrahomogeneous} and hence the dynamical properties  $\mathrm{Aut}(\mathcal{N})$ can be reflected to the combinatorial properties of the class $\mathrm{Age}(\mathcal{N})$,  consisting  of all finitely generated structures that embed in $\mathcal{N}$. We briefly review the pertinent definitions and some relevant examples below. For more details one may consult  \cite{Hodges} or \cite{TZ12}.

Let $\mathcal{L}$ be a first order language. We say that  $\mathcal{L}$ is {\bf relational} if it only contains relation symbols. If $\mathcal{M}$ is an $\mathcal{L}$-structure then we write $\mathcal{M}=(M,\ldots)$ to indicate that the domain of  $\mathcal{M}$ is the set $M$. 
Let $\mathcal{N}=(\mathbb{N},\ldots)$ be an $\mathcal{L}$-structure.  For every $A\subseteq \mathbb{N}$ we denote by $\langle A\rangle_{\mathcal{N}}$ the smallest substructure of $\mathcal{N}$  containing every element of $A$ in its domain. We say that $\mathcal{N}$ is {\bf ultrahomogeneous} if for every bijection $f\colon A\to B$ between finite subsets of $\mathbb{N}$ which induces an isomorphism from   $\langle A\rangle_{\mathcal{N}}$  to  $\langle B\rangle_{\mathcal{N}}$,  we have that $f$ extends to an automorphism of $\mathcal{N}$. We denote by $\mathrm{Age}(\mathcal{N})$ the class of all finitely generated $\mathcal{L}$-structures which embed into $\mathcal{N}$. If $\mathcal{N}$ is ultrahomogeneous, then $\mathcal{K}:=\mathrm{Age}(\mathcal{N})$ satisfies the {\bf amalgamation property}, i.e., for every two embeddings  $f\colon \mathcal{A}\to \mathcal{B}$, $g\colon \mathcal{A}\to \mathcal{C}$, with $\mathcal{A},\mathcal{B},\mathcal{C}\in \mathcal{K}$, there exists $\mathcal{D}\in \mathcal{K}$ and embeddings  $f'\colon \mathcal{B}\to \mathcal{D}$, $g'\colon \mathcal{C}\to \mathcal{D}$ so that $g'\circ g= f'\circ f$.
When $\mathcal{N}$ is ultrahomogeneous, then various strengthenings of the amalgamation property, as well as several other combinatorial properties of $\mathcal{K}$, often correspond to dynamical properties of $P:=\mathrm{Aut}(\mathcal{N})$. We recall here some  amalgamation properties of $\mathcal{K}$ which relate to the dynamical properties of $P$ mentioned in the introduction such as $P$  being locally finite or $n$-free.

Let $\mathcal{N}=(\mathbb{N},\ldots)$ be an  ultrahomogeneous $\mathcal{L}$-structure and let  $\mathcal{K}:=\mathrm{Age}(\mathcal{N})$. We say that $\mathcal{K}$ satisfies the {\bf disjoint amalgamation property}---also known as {\bf strong amalgamation property}---if it satisfies the amalgamation property, as stated above, but in addition to $g'\circ g= f'\circ f$, the maps $f,f',g,g'$ have to also satisfy: $\mathrm{range}(f')\cap\mathrm{range}(g')=\mathrm{range}(f'\circ f)$. It is well known that $\mathcal{K}$ has the disjoint amalgamation property if and only if $[A]_P=\mathrm{dom}(\langle A\rangle_{\mathcal{N}})$, for every finite $A\subseteq \mathbb{N}$---this is a classical application of  Neumann's lemma \cite{neumann1976structure}*{Lemma 2.3} see \cite{Hodges}*{Exercise 7.1.8.}. As a consequence, if $\mathcal{L}$ is relational, then the  disjoint amalgamation property for $\mathcal{K}$ implies that $P$ satisfies a well-studied strengthening of local finiteness known as {\bf no algebraicity}: $[A]_P=A$, for all finite $A\subseteq \mathbb{N}$. Similarly, if $\mathcal{K}$ consists only of finite (possibly non-relational) structures, then the  disjoint amalgamation property for $\mathcal{K}$ implies that $P$ is locally finite. In fact, it is not difficult to see that, when $\mathcal{K}$ consists only of finite structures, then $P$ is locally finite if and only if $\mathcal{K}$ has the {\bf cofinal disjoint amalgamation property}, i.e., if there is some $\mathcal{K}_0\subseteq \mathcal{K}$ which has the disjoint amalgamation property so that every structure in  $\mathcal{K}$ embeds to some structure in   $\mathcal{K}_0$;  see \cite{Random}*{Theorem 4.14.}. 

Next we recall the notion of \emph{disjoint $n$-amalgamation}, for $\mathcal{K}=\mathrm{Age}(\mathcal{N})$ as above, which relates to the property of $P=\mathrm{Aut}(\mathcal{N})$ being $n$-free---here we follow \cite{KP} whose notation is closer to the one in this paper. Let $\Delta^{n-1}$ be the poset $(\mathcal{P}(n),\subseteq)$ of all subsets of $n:=\{0,1,\ldots,n-1\}$ ordered by inclusion. We denote by $\partial \Delta^{n-1}$ the subposet of $\Delta^{n-1}$ consisting of all but the top element $n\in \Delta^{n-1}$.
An {\bf $n$-cube in $\mathcal{K}$} is a functor\footnote{As usual we view $\mathcal{K}$ as a category whose morphisms are embeddings of $\mathcal{L}$-structures.} $T$ from $\Delta^{n-1}$ to $\mathcal{K}$, i.e., a pair $T= \big((\mathcal{A}_{\sigma})_{\sigma},(f^{\sigma}_{\tau})_{\sigma\subseteq\tau} \big)$, where $(\mathcal{A}_{\sigma})_{\sigma}$ is a collection of structures from $\mathcal{K}$, indexed by elements $\sigma$ of $\Delta^{n-1}$; together with an embedding $f^{\sigma}_{\tau}\colon  \mathcal{A}_{\sigma}\to \mathcal{A}_{\tau}$, whenever $\sigma\subseteq \tau$, so that each $f^\sigma_\sigma$ is the identity map, and for every $\sigma\subseteq\tau\subseteq\rho$ we have that $f^{\tau}_{\rho} \circ f^{\sigma}_{\tau}=f^{\sigma}_{\rho}$. Similarly, a \textbf{partial $n$-cube} in $\mathcal{K}$ is a functor from $\partial\Delta^{n-1}$ to $\mathcal{K}$. An $n$-cube in $\mathcal{K}$ is \textbf{disjoint} if
\begin{align}\label{Eq:Disjoint}
\mathrm{range}(f^\sigma_{\sigma\cup \tau})\cap \mathrm{range}(f^\tau_{\sigma\cup \tau}) = \mathrm{range}(f^{\sigma\cap \tau}_{\sigma\cup\tau})    
\end{align}
for all $\sigma$ and $\tau$. Similarly, a partial $n$-cube is disjoint if the same condition holds whenever $\sigma\cup \tau\subsetneq n$. The class $\mathcal{K}$ has \textbf{disjoint $n$-amalgamation} if every disjoint partial $n$-cube can be extended to a disjoint $n$-cube. %We will be interested in disjoint $n$-amalgamation when $n\geq 2$.   
Notice that disjoint $2$-amalgamation coincides to disjoint amalgamation as defined in the previous paragraph. Examples of ultrahomogeneous structures $\mathcal{N}_n$ so that $\mathrm{Age}(\mathcal{N}_n)$ has disjoint $m$-amalgamation for all $m\leq n$ but  does not have disjoint $(n+1)$-amalgamation were first constructed in \cite{BKL17}; see Example \ref{Ex:2} below. 

Similarly one defines the \emph{disjoint $n$-joint embedding property}: let $\Delta^{n-1}_{-}$  and $\partial\Delta^{n-1}_{-}$ be the subposets of $\Delta^{n-1}$ and $\partial\Delta^{n-1}$ respectively that contain all elements therein but the empty set $\emptyset\in \mathcal{P}(n)$. We say that $\mathcal{K}$ has the {\bf disjoint $n$-joint embedding property} if every functor $\partial\Delta^{n-1}_{-}\to \mathcal{K}$
satisfying (\ref{Eq:Disjoint}) above for all  $\sigma,\tau \in \partial\Delta^{n-1}_{-}$  extends to a functor $\Delta^{n-1}_{-}\to \mathcal{K}$ which  satisfies (\ref{Eq:Disjoint}) for all $\sigma,\tau \in \Delta^{n-1}_{-}$. Here, we take $\mathrm{range}(f^{\sigma\cap \tau}_{\sigma\cup\tau})$ to simply be the empty set whenever $\sigma\cap\tau=\emptyset$. When the language $\mathcal{L}$ contains no constant symbols, then the empty structure is an $\mathcal{L}$-structure and hence the disjoint $n$-amalgamation property implies the disjoint $n$-joint embedding property. However this is not true in general.

\begin{lemma}\label{L:new}
Let $\mathcal{K}:=\mathrm{Age}(\mathcal{N})$ where $\mathcal{N}=(\mathbb{N},\ldots)$ is an ultrahomogeneous structure. If $\mathcal{K}$ satisfies the disjoint amalgamation property and the  disjoint $k$-joint embedding property for all $k\leq n$, then $P:=\mathrm{Aut}(\mathcal{N})$ is $n$-free.
\end{lemma}
\begin{proof}
Let $F\subseteq \mathbb{N}$ be finite and set $\mathcal{A}:=\langle F \rangle_{\mathcal{N}}\in\mathcal{K}$.
We will first define for each $k\in\{1,\ldots,n\}$ functors  
$T_k= \big((\mathcal{A}_{\sigma})_{\sigma},(f^{\sigma}_{\tau})_{\sigma\subseteq\tau} \big)$ from  $\Delta^{k-1}_{-}$ to $\mathcal{K}$ so that:
\begin{enumerate}
    \item[(i)] $T_{k+1}\upharpoonright \Delta^{k-1}_{-}=T_k$, when $k<n$;
    \item[(ii)] $\mathcal{A}_{\{0\}}=\mathcal{A}$, and $\mathcal{A}_\sigma$ is isomorphic to $\mathcal{A}_{\sigma'}$ whenever $\sigma,\sigma'$ are of the same size;
    \item[(iii)] for all $\sigma,\tau \in \Delta^{k-1}_{-}$ the identity (\ref{Eq:Disjoint}) holds, where $\mathrm{range}(f^{\sigma\cap \tau}_{\sigma\cup\tau})$ is taken to be the empty set whenever $\sigma\cap\tau=\emptyset$.
\end{enumerate}

This is done by induction on $k$. For $k=1$ we simply let $\mathcal{A}_{\{0\}}:=\mathcal{A}$. Assume now that we have defined $T_{k-1}$ by providing the necessary data.  Namely, $\mathcal{A}_{\sigma}$ and $f^{\sigma}_{\tau}$ for all $\sigma,\tau\subseteq \{0,\ldots,k-2\}$ with $\sigma\subseteq \tau$, so that (ii) and (iii) above hold. We may extend this functor to a functor $T^{\partial}_k\colon \partial  \Delta^{k-1}_{-}\to \mathcal{K}$ as follows. If $0\leq \ell_0<\ldots<\ell_i=k-1$  with $i\neq k-1$, then set $\mathcal{A}_{\{\ell_0,\ldots,\ell_i\}}$ to be an isomorphic copy of $\mathcal{A}_{\{0,\ldots,i\}}$ whose domain is disjoint from all the domains of all structures so far involved in the ongoing construction; let $\varphi_{\{\ell_0,\ldots,\ell_i\}}\colon \mathcal{A}_{\{\ell_0,\ldots,\ell_i\}} \to \mathcal{A}_{\{0,\ldots,i\}}$ be the associated isomorphism. If in addition we have  $0\leq m_0<\ldots<m_j\leq k-1$ with $j\neq k-1$, then set 
\[f_{\{\ell_0,\ldots,\ell_i\}}^{\{m_0,\ldots,m_j\}}:=\varphi^{-1}_{\{\ell_0,\ldots,\ell_i\}}\circ f_{\{0,\ldots,i\}}^{\{0,\ldots,j\}} \circ \varphi_{\{m_0,\ldots,m_j\}},\] 
where $\varphi_{\{m_0,\ldots,m_j\}}=\mathrm{id}$ if $m_j<k-1$.
It is easy to see that this induces a functor $T^{\partial}_k\colon \partial  \Delta^{k-1}_{-}\to \mathcal{K}$ so that (i),(ii),(iii) hold when $\sigma,\sigma',\tau$ range in $\partial  \Delta^{k-1}_{-}$. Since $\mathcal{K}$ has the $(k-1)$-joint embedding property, the functor $T^{\partial}_k$ extends to a functor $T_k\colon   \Delta^{k-1}_{-}\to \mathcal{K}$  so that (i),(ii),(iii) hold.

To conclude with the proof, let $\mathcal{B}:=\mathcal{A}_{\{0,\ldots,n-1\}}\in\mathcal{K}$ be the top element of the image of $T_{n}$ and for every $i<n$ set $h_i:=f^{\{i\}}_{\{0,\ldots,n-1\}}$. By (ii),(iii) above, for each $i<n$, the map $h_{i}$ is an embedding from $\mathcal{A}$ to $\mathcal{B}$ so that
\begin{align}\label{Eq:Disjoint2}
\mathrm{range}(h_i)\cap \mathrm{dom}(\langle \bigcup_{j \colon j\neq i}\mathrm{range}(h_j)\rangle_{\mathcal{B}})=\emptyset. 
\end{align}
Since $\mathcal{B}\in\mathcal{K}$, we may assume without loss of generality that $\mathcal{B}$ is a substructure of $\mathcal{N}$. Since $\mathcal{N}$ is ultrahomogeneous and $\mathcal{A}$ is generated by the finite $F$,  we can extend $h_0,\ldots,h_{n-1}$ into automorphisms $g_0,\ldots,g_{n-1}\in P$. Notice that, by (\ref{Eq:Disjoint2}) above, for all $i\in\{0,\ldots, n-1\}$ we have that $\mathrm{dom}(\langle g_i\cdot F\rangle_{\mathcal{N}})$ and    $\mathrm{dom}(\langle \bigcup_{j : j\neq i} (g_j\cdot F)\rangle_{\mathcal{N}})$ are disjoint. Since $\mathcal{K}$ has the disjoint amalgamation property, the latter is equivalent to $[g_i\cdot F]_P$ and    $[\bigcup_{j : j\neq i} (g_j\cdot F)]_{P}$ being disjoint, concluding the proof that  $P$ is $n$-free.
\end{proof}

\subsection{The model-theoretic algebraic closure}

If $\mathcal{N}$ is an $\mathcal{L}$-structure, then the {\bf $\mathcal{L}$-algebraic closure} $\mathrm{acl}_{\mathcal{N}}(A)$ of $A\subseteq\mathrm{dom}(\mathcal{N})$  is the set of all $b\in \mathrm{dom}(\mathcal{N})$ for which there is an $\mathcal{L}$-formula $\varphi$, with parameters from $A$, so that $\mathcal{N}\models\varphi(b)$ and $\{c\in \mathrm{dom}(\mathcal{N}) \colon \mathcal{N}\models\varphi(c)\}$ is finite. Similarly to Definition \ref{Def1}, one can use the operator $A\mapsto\mathrm{acl}_{\mathcal{N}}(A)$ to define a notion of dimension for $\mathcal{L}$-structures. Namely:

\begin{definition}\label{def:dim_for_structures}
The {\bf algebraic dimension} $\mathrm{dim}_{\mathcal{L}}(\mathcal{N})$ of an $\mathcal{L}$-structure $\mathcal{N}$ is the largest $n\in\mathbb{N}\cup\{\infty\}$ for which there exists $A\subseteq \mathbb{N}$, with $|A|=n$, so that for all $a\in A$ we have that $a\not\in \mathrm{acl}_{\mathcal{N}}(A\setminus \{a\})$. 
\end{definition}

In general, if $P:=\mathrm{Aut}(\mathcal{N})$, then the algebraic closure $A\mapsto [A]_P$ defined in the introduction does not coincide with its model-theoretic counterpart $A\mapsto\mathrm{acl}_{\mathcal{N}}(A)$. In particular, the algebraic dimensions $\mathrm{dim}_{\mathcal{L}}(\mathcal{N})$ and $\mathrm{dim}(P)$ may not coincide. However, we can always enrich the structure $\mathcal{N}$ and make the two notions equal:

\begin{lemma}\label{PermutationUltrahomogeneous}
If $P$ is a Polish permutation group, then there is exists ultrahomogeneous $\mathcal{L}$-structure $\mathcal{N}=(\mathbb{N},\ldots)$, on a countable relational language $\mathcal{L}$, so that $P=\mathrm{Aut}(\mathcal{N})$, and with $[A]_P=\mathrm{acl}_{\mathcal{N}}(A)$, for every finite $A\subseteq \mathbb{N}$.  
\end{lemma}
\begin{proof} This follows from the usual ``orbit completion" argument; e.g. see \cite{Hodges}*{Theorem 4.1.4} or \cite{BK96}*{Section 1.5}: for every $n\in 1,2,\ldots$, and for every orbit $O\subseteq \mathbb{N}^n$ in the diagonal action $g\cdot(a_1,\ldots,a_n)=(g\cdot a_1,\ldots,g\cdot a_n)$ of $P$ on $\mathbb{N}^n$, we introduce a $n$-ary relation symbol $R^{n,O}$ and we set $\mathcal{N}\models R^{n,O}(\bar{a})$ if and only if $\bar{a}\in O$. Clearly $\mathcal{N}$ is ultrahomogeneous and  $P=\mathrm{Aut}(\mathcal{N})$ follows from the fact that $P$ a closed subgroup of $\mathrm{Sym}(\mathbb{N})$.  It is also clear that    $[A]_P=\mathrm{acl}_{\mathcal{N}}(A)$ holds for all finite $A\subseteq \mathbb{N}$. Indeed, if $\bar{a}\in\mathbb{N}^n$  and $b\in\mathbb{N}$, then  the orbit of $b$ under $P_{\bar{a}}$  is finite if and only if $\{c\in\mathbb{N}\colon \mathcal{N}\models\varphi(\bar{a},c) \}$ is finite,  where $\varphi(\bar{x},y)$ is the formula  $R^{n+1,O}(\bar{x},y)$ associated to    the orbit $O$ of $\bar{a}b\in \mathbb{N}^{n+1}$ under $P$.
\end{proof}

\subsection{Examples} The following examples will be important in what follows.

\begin{example}\label{Ex:1}
If $P=\mathrm{Sym}(\mathbb{N})$ is the full symmetric group, then $E_{\mathrm{inj}}(\mathrm{Sym}(\mathbb{N}))$ is Borel bireducible  with the equivalence relation $=^{+}$ of {\bf countable sets of reals}:   $(x_n) =^{+} (y_n)$ if and only if $\{x_n:n\in\mathbb{N}\}=\{y_n: n\in \mathbb{N}\}$, for all $(x_n), (y_n)\in \mathbb{R}^{\mathbb{N}}$. 
\end{example}

\begin{example}\label{Ex:2}
Fix $n\geq 2$ and consider the language $\mathcal{L}_n$ which contains countably many $n$-ary relation symbols $R_0,R_1,R_2,\ldots$  and countably many $n$-ary function symbols $f_0,f_1,f_2,\ldots$. Consider the collection of all countable structures $\mathcal{M}$ which satisfy the following properties:
\begin{enumerate}
    \item $\mathcal{M}$ is locally finite, i.e., $\mathrm{dom}(\langle A\rangle_{\mathcal{M}})$ is finite if so is $A\subseteq \mathrm{dom}(\mathcal{M})$;
    \item the collection $\{R^{\mathcal{M}}_j\colon j \in\mathbb{N} \}$ partitions $\mathrm{dom}(\mathcal{M})^n$;
    \item if $R^{\mathcal{M}}_j(a_1,\cdots, a_n)$, then for all $i>j$ we have $f^{\mathcal{M}}_i(a_1,\cdots, a_n)=a_1$;
    \item if $A\subseteq \mathrm{dom}(\mathcal{M})$ is of size $n+1$, then there is $a\in A$ with $a\in \mathrm{dom}(\langle A\setminus \{a\}\rangle_{\mathcal{M}})$.
\end{enumerate}
By \cite{BKL17}*{Corollary 3.2},  there is a unique up to isomorphism ultrahomogenous countable $\mathcal{L}_n$-structure $\mathcal{N}$ satisfying (1)--(4) that embeds all finite $\mathcal{L}_n$-structures  satisfying (2)--(4).
From now on, we fix an isomorphic copy of this structure on domain $\mathbb{N}$ and we denote it by $\mathcal{N}^{\mathrm{BKL}}_n$. We let  $P_n:=\mathrm{Aut}(\mathcal{N}^{\mathrm{BKL}}_n)$ and $\mathcal{K}_n:=\mathrm{Age}(\mathcal{N}^{\mathrm{BKL}}_n)$. It turns out that  the class $\mathcal{K}_n$ has the disjoint  $k$-amalgamation for all $k\leq n$---see \cite{BKL17}*{Theorem 3.1}, or  \cite{KP}*{Lemma 12.(1)} for a treatment closer in spirit to the one here. In particular, $\mathcal{K}_n$ has the  
 disjoint amalgamation property and, as $\mathcal{L}_n$ has no constants, the  disjoint $k$-joint embedding property for all $k\leq n$; see Section \ref{SS:BS'}.  It follows by Lemma \ref{L:new} that $P_n$ is $n$-free when $n\geq 2$. As it is always true that 
 \[\mathrm{dom}(\langle A\rangle_{\mathcal{N}^{\mathrm{BKL}}_n}) \subseteq \mathrm{acl}_{\mathcal{N}^{\mathrm{BKL}}_n}(A) \subseteq  [A]_{P_n},\]
 for all finite $A\subseteq \mathbb{N}$,
 the fact that  $\mathcal{K}_n$  has the disjoint amalgamation property forces all three sets above to be equal; see Section \ref{SS:BS'}. Since $\mathcal{K}_n$ consists only of finite structures, this implies that $P_n$ is a locally  finite permutation group. Finally, Property (4) above forces the algebraic dimension of $P_n$ to be at most $n$, and since $P_n$ is $n$-free we have $\mathrm{dim}(P_n)=n$. 
The Bernoulli shifts of $P_n=\mathrm{Aut}(\mathcal{N}^{\mathrm{BKL}}_n)$  induce  the following collection of equivalence relations which, in view of  Theorem \ref{T:1}, will play an important role in what follows:
\[E(\mathrm{Aut}(\mathcal{N}^{\mathrm{BKL}}_2)), E(\mathrm{Aut}(\mathcal{N}^{\mathrm{BKL}}_3)),\ldots, E(\mathrm{Aut}(\mathcal{N}^{\mathrm{BKL}}_n)), \ldots,\]
\end{example}

\subsection{Labelled logic actions}\label{SS:LabelledAction}
Let $\mathcal{L}$ be a countable language and let  $\mathrm{Str}(\mathcal{L},\mathbb{N})$ be the Polish space of  all $\mathcal{L}$-structures on domain $\mathbb{N}$; see \cites{Kec95,Gao09}. The {\bf logic action} is the action $(g,\mathcal{N})\mapsto g\mathcal{N}$ of $\mathrm{Sym}(\mathbb{N})$ on $\mathrm{Str}(\mathcal{L},\mathbb{N})$, which is  defined by setting for all relations $R\in\mathcal{L}$  and all functions $f\in\mathcal{L}$:
\begin{align*}
g\mathcal{N}\models R(n_0,\ldots,n_{k-1})\quad &\iff\quad \mathcal{N}\models R(g^{-1}(n_0),\ldots,g^{-1}(n_{k-1})) \\
g\mathcal{N}\models f(n_0,\ldots,n_{k-1})=n\quad &\iff\quad \mathcal{N}\models f(g^{-1}(n_0),\ldots,g^{-1}(n_{k-1}))=g^{-1}(n) 
\end{align*}
If $\sigma$ is an $\mathcal{L}_{\omega_1,\omega}$-sentence then we will denote by $\mathrm{Str}(\sigma,\mathbb{N})$ the $\mathrm{Sym}(\mathbb{N})$-invariant Borel subset of $\mathrm{Str}(\mathcal{L},\mathbb{N})$ which consists of all $\mathcal{N}$ with $\mathcal{N}\models\sigma$. 

As in \cite{KP}, we also consider the {\bf labelled logic action} $\mathrm{Sym}(\mathbb{N})\curvearrowright \mathrm{Str}_{\mathbb{R}}(\mathcal{L},\mathbb{N})$, where 
$\mathrm{Str}_{\mathbb{R}}(\mathcal{L},\mathbb{N}):=\mathbb{R}^{\mathbb{N}}\times\mathrm{Str}(\mathcal{L},\mathbb{N})$ and  $(g,(x,\mathcal{N}))\mapsto (gx,g\mathcal{N})$. The {\bf injective part  of the labelled logic action} is the  restriction of the above action to 
$\mathrm{Sym}(\mathbb{N})$-invariant subset 
 $\mathrm{Str}^{\mathrm{inj}}_{\mathbb{R}}(\mathcal{L},\mathbb{N})$ of $\mathrm{Str}_{\mathbb{R}}(\mathcal{L},\mathbb{N})$, given by
\[\mathrm{Str}^{\mathrm{inj}}_{\mathbb{R}}(\mathcal{L},\mathbb{N}):= \mathrm{Inj}(\mathbb{N},\mathbb{R})\times \mathrm{Str}(\mathcal{L},\mathbb{N}).\]
Notice that $\mathrm{Str}^{\mathrm{inj}}_{\mathbb{R}}(\mathcal{L},\mathbb{N})$ is a dense $G_{\delta}$ subset of $\mathrm{Str}_{\mathbb{R}}(\mathcal{L},\mathbb{N})$.

Similarly we define the  labelled logic action $\mathrm{Sym}(\mathbb{N})\curvearrowright \mathrm{Str}_{\mathbb{R}}(\sigma,\mathbb{N})$ on  models of an $\mathcal{L}_{\omega_1,\omega}$-sentence $\sigma$, together with its injective part $\mathrm{Str}^{\mathrm{inj}}_{\mathbb{R}}(\sigma,\mathbb{N})$.
We will denote by $\simeq_{\mathrm{iso}}$ the associated orbit equivalence relation.
For every $\mathcal{L}$-structure $\mathcal{N}=(\mathbb{N},\ldots)$, the pair $(i,f)$ where  $i\colon \mathrm{Aut}(\mathcal{N})\to \mathrm{Sym}(\mathbb{N})$ and  $f\colon \mathbb{R}^{\mathbb{N}}\to \mathrm{Str}_{\mathbb{R}}(\mathcal{L},\mathbb{N})$, with  $i(g)=g$ and $f(x)=(x,\mathcal{N})$, forms  an $\mathrm{Aut}(\mathcal{N})$-equivariant embedding  of the Bernoulli shift:
\begin{equation*}
\big(\mathrm{Aut}(\mathcal{N})\curvearrowright \mathbb{R}^{\mathbb{N}} \big)\hookrightarrow
\big(\mathrm{Sym}(\mathbb{N})\curvearrowright \mathrm{Str}_{\mathbb{R}}(\mathcal{L},\mathbb{N})\big).
\end{equation*}
That is, for all $g\in \mathrm{Aut}(\mathcal{N})$ and $x\in \mathbb{R}^{\mathbb{N}}$ we have that $i(g)\cdot f(x)=f(g\cdot x)$. As a consequence, $f\colon \mathbb{R}^{\mathbb{N}}\to \mathrm{Str}_{\mathbb{R}}(\mathcal{L},\mathbb{N})$  induces a Borel reduction:
\begin{equation*}
E(\mathrm{Aut}(\mathcal{N}))\leq_B \big(\mathrm{Str}_{\mathbb{R}}(\mathcal{L},\mathbb{N}), \simeq_{\mathrm{iso}}\big).
\end{equation*}
Notice that the above map embeds the injective part of the Bernoulli shift into the injective part of the labelled logic action.

\section{A symmetric model argument}\label{sec: symmetric model}

One of the many early consequences of forcing was the independence of the axiom of choice (AC) from the  Zermelo–Fraenkel (ZF) axioms. To construct a model of ZF+$\neg$AC, Cohen started with the Cohen forcing $\mathbb{P}$ which adds a countable sequence of Cohen reals $\left<x_n\colon n\in \mathbb{N}\right>$, but instead of taking the full forcing extension $V[G]$, he considered the substructure of $V[G]$ generated by all $\mathbb{P}$-names which are invariant under the natural action of $\mathrm{Sym}(\mathbb{N})$ on $\mathbb{P}$ by index permutation; see \cite{Ka08}. The resulting  {\bf basic Cohen model}  coincides with the smallest extension $V(\{x_n\colon n\in \mathbb{N}\})$ of $V$ in $V[G]$ which contains the set $\{x_n\colon n\in \mathbb{N}\}$ and  satisfies ZF. The failure of AC in this model comes from the fact that $V(\{x_n\colon n\in \mathbb{N}\})$  contains no well-ordering of $\{x_n\colon n\in \mathbb{N}\}$. The latter is established by developing a \emph{theory of supports} for definable sets in $V(\{x_n\colon n\in \mathbb{N}\})$. 

Since $\{x_n\colon n\in \mathbb{N}\}$ can be alternatively thought as the classifying $=^{+}$-invariant of the generic point $(x_n)\in \mathbb{R}^{\mathbb{N}}$, one would hope that the ergodic theoretic properties of $=^{+}$ are reflected to the structure of definable sets in  $V(\{x_n\colon n\in \mathbb{N}\})$. This point of view is adopted in  \cites{Sha18,Sha19}, where a general translation principle  is established between strong ergodic properties of  equivalence relations which are classifiable by countable structures and the structure of definable sets in the associated \emph{symmetric models}; see Lemma \ref{lem;generic-erg-symm-model} below. 

In this section we build further on the symmetric model techniques developed in  \cites{Sha18,Sha19} and use the resulting theory to prove Theorem \ref{T:1} and Corollary \ref{cor:1} from the introduction. In the process we
develop a robust theory of supports,  similar to the one in the basic Cohen model, for symmetric models which are associated to locally finite permutation groups; see Lemma \ref{lem;supports-for-V(A)}. 

\subsection{Symmetric models}
For the rest of this section, and through the paper, we work over the universe $V=L$.
This can be done as all the equivalence relations we talk about are in $L$, and the questions of Borel reducibility and strong ergodicity between them are absolute.
This restriction to $L$ is not necessary, and will never in fact be used. For this reason we always use the notation $V$ instead of $L$. We only do this to simplify the introduction below to ordinal definability and minimal extensions, and provide a more concrete context for the reader who is less familiar with the subject.

Let $V[G]$ be a generic extension of $V$ with respect to some forcing notion $\mathbb{P}$  and let $A\in V[G]$. Let $V(A)$ be the smallest transitive substructure $W$ of $V[G]$ such that $W$ satisfies ZF, $W$ extends $V$, and $W$ contains $A$. This is precisely the Hajnal relativized $L$-construction $L(A)$; see \cite{Jec03}*{p. 193}. Note that $L(A)$ depends only on $A$, and not on the poset $\mathbb{P}$ or the ambient universe $V[G]$.\footnote{Without the assumption $V=L$, $V(A)$ can be similarly defined as the union of all models of the form $L(A,v)$ for some $v\in V$, and it depends only on $V$ and $A$.} We generally refer to such models $V(A)$ as \textbf{symmetric models}.

Recall the  notion of ordinal definability  \cite{Jec03}*{p. 194}:
if $W$ is a model of ZF and $A$ is a set in $W$, then  $S$ is \textbf{ordinal definable over $A$} (in $W$), denoted $S\in\mathrm{OD}(A)^W$, if $S\in W$ and there is a formula $\varphi$, ordinal parameters $\bar{b}$ and parameters $\bar{a}$ from the transitive closure of $A$, so that $S=\{s\in W: W\models \varphi(s,A,\bar{a},\bar{b})\}$. Here we will be working with models $W$ of ZF which are generic extensions of $V$, or intermediate extensions of such, and therefore they all have the same ordinals with $V$.

Let $\mathrm{HOD}(A)^{W}$ be the subclass of $W$ consisting of all sets $S\in W$ so that every element of the { transitive closure  $\mathrm{tc}(\{S\})$ of $\{S\}$} is in $\mathrm{OD}(A)^{W}$.
If $W$ is any model of ZF and $A$ a set in $W$, then $\mathrm{HOD}(A)^{W}$ is a transitive model of ZF. See the proof of \cite{Jec03}*{Theorem 13.26} and the following discussion.

\begin{fact}\label{Fact:1}
For every $S\in V(A)$ there is some formula $\varphi$, parameters $\bar{a}$ from the transitive closure of $A$ and $v\in V$ such that $S=\{s\in V(A) : V(A) \models \varphi(s,A,\bar{a},v) \}$.
In this case we say that \textbf{$\bar{a}$ is a definable support for $S$}.
\end{fact}
\begin{proof}
Work inside $V(A)$, and consider $\mathrm{HOD}(A)^{V(A)}$, the class of sets which are hereditarily ordinal definable over $A$, as calculated in $V(A)$.
Then $\mathrm{HOD}(A)^{V(A)}$ is a transitive model of ZF containing $A$ and all the ordinals.
By the minimality of $V(A)$, it follows that $V(A)$ is equal to $\mathrm{HOD}(A)^{V(A)}$. 
So every set in $V(A)$ is definable as in the statement of the fact.\footnote{If we were not to assume that $V=L$, the same arguments hold when $\mathrm{HOD}(A)$ is changed to $\mathrm{HOD}(A,V)$, all the sets hereditarily definable using parameters from the transitive closure of $A$ and arbitrary parameters in $V$.}
\end{proof}

\subsection{Absolute classification}
In what follows, we study the Borel complexity of various analytic equivalence relations $E$, such as the orbit equivalence relations of Bernoulli shifts $E(P)$, by analysing models of the form $V(A)$ where $A$ is a classifying invariant for $E$. This is possible when $E$ admits a ``reasonable enough" classification, as follows.
%In what follows, we are going to use symmetric models to study the Borel complexity of various analytic equivalence relations $E$, such as the orbit equivalence relations of Bernoulli shifts $E(P)$. For this it will be important that all equivalence relations $E$ we consider admit a ``reasonable enough" classification by hereditarily countable sets.

Let $X\in V$ be a Polish space and  $E\in V$  an analytic equivalence relation on $X$. A {\bf complete classification of $(X,E)$ in $V$} is any assignment $x\mapsto A_x$ in $V$ given by a first-order formula $\varphi(x,A_x)$ of set theory, so that for all $x,y\in X$ we have:
\[x\mathrel{E}y \iff A_x=A_y.\]
We write $x\mapsto_{\varphi} A_x$ when the assignment  $x\mapsto A_x$ is given by $\varphi$. The most ``canonical" such assignment would be to map each $x$ to its $E$-equivalence class $[x]_E$. While this assignment is canonical in that it is defined uniformly for all such pairs $(X,E)$, it does not give us any new information about the classification problem  $(X,E)$ and it does not behave nicely when passing to forcing extensions. For example,  even when $x\in X\cap V$, $[x]_E$ as computed in $V[G]$ may not coincide with $[x]_E$ as computed in $V$, since the forcing may add new members in the $E$-class of $x$. A complete classification
$x\mapsto_{\varphi} A_x$  is an {\bf absolute classification} if whenever $W\supseteq N \supseteq  V$ are models of ZF, the assignment $x\mapsto_{\varphi} A_x$, as computed in both $W$ and $N$ respectively, is still a complete classification for $(X,E)$, and for all $x\in X\cap N$ we have that $A_x$ as computed in $W$ is equal to  $A_x$ as computed in $N$. 

%\begin{enumerate}
%    \item $\varphi(x,A_x)$ implies that $A_x$ is a hereditarily countable set;
 %   \item if $W\supseteq V$ is a model of ZF then $x\mapsto_{\varphi} A_x$ is still a complete classification for $(X,E)$ in $W$ and for all $x\in X\cap V$ we have that $A_x$ as computed in $W$ is  equal to  $A_x$ as computed in $V$. 
%\end{enumerate}

\begin{example}\label{Ex:11}
Recall $E(\mathrm{Sym}(\mathbb{N}))$ from Example \ref{Ex:1}. The assignment 
\begin{align*}
\left<x_n:\,n\in\mathbb{N}\right>&\mapsto\{x_n:\,n\in\mathbb{N}\}    
\end{align*}
defined on $\mathbb{R}^\mathbb{N}$ is a complete classification when restricted to the comeager subset $\mathrm{Inj}(\mathbb{N},\mathbb{R})$ of $\mathbb{R}^{\mathbb{N}}$, consisting of all injective functions from $\mathbb{N}$ to $\mathbb{R}$. 
It is easy to check that it is also an absolute classification. 
\end{example}

More generally, if $(X,E)$ is classifiable by countable structures, then it also admits an absolute classification $x\mapsto A_x$, where $A_x$ is a \emph{hereditarily countable set}. Such absolute classification is attained simply by post-composing the Borel reduction which witnesses classifiability by countable structures with the usual Scott analysis procedure; see \cites{Gao09,Hjo00}. Being classifiable by countable structures is an absolute condition. For the absoluteness of the Scott analysis see also  \cite{Fri00}*{Lemma 2.4}.

The following theorem establishes the main connection between Borel reduction complexity and symmetric models which we are going to use later on. 
Notice that in the Lemma below, if $A_x$ is computed with respect to the assignment of Example \ref{Ex:11}, where $x\in\mathbb{R}^{\mathbb{N}}$ is Cohen generic, then $V(A_x)$  coincides with the basic Cohen model.

\begin{lemma}[\cite{Sha19}]\label{lem;generic-erg-symm-model}
Suppose $E$ and $F$ are Borel equivalence relations on Polish spaces $X$ and $Y$ respectively and $x\mapsto A_x$ and $y\mapsto B_y$ are absolute classifications of $E$ and $F$ respectively. Then,
the following are equivalent.
\begin{enumerate}
    \item For every  Borel homomorphism  $f\colon (X_0,E)\to (Y,F)$, where $X_0\subseteq X$ is non-meager, $f$ maps a non-meager set into a single $F$-class;
    \item If $x\in X$ is Cohen-generic over $V$ and $B$ is a potential $F$-invariant in $V(A_x)$ which is definable only from $A_x$ and parameters in $V$, then $B\in V$.
\end{enumerate}
\end{lemma}

Here, by ``{\bf $\mathbf{B}$ is a potential $\mathbf{F}$-invariant}'' we simply mean that there is a $y$ in some further generic extension, such that $B=B_y$. For example, any set of reals in $V(A_x)$  is a potential $=^+$-invariant for the assignment from Example \ref{Ex:11}, since we can always move to a forcing extension which collapses it to a countable set.

\begin{remark}
Suppose $E$ is generically ergodic. Then condition (1) in Lemma~\ref{lem;generic-erg-symm-model} is equivalent to: 
$E$ is generically $F$-ergodic.
In our examples below $E$ will be generically ergodic and Lemma~\ref{lem;generic-erg-symm-model} will be used to conclude generic $F$-ergodicity.
\end{remark}
We sketch the proof idea of Lemma~\ref{lem;generic-erg-symm-model} for the reader who is familiar with Cohen forcing and countable substructures. 
See for example \cite[Section~26]{Jec03}. We provide below an explicit definition of the forcing when $X=\mathbb{R}^{\mathbb{N}}$, which is the main example used in this paper.
\begin{proof}[Proof sketch of Lemma~\ref{lem;generic-erg-symm-model} (2)$\implies$(1)]
We briefly sketch the proof of the direction which we are going to use  in this paper.
Assume (2), and let $f$ be as in (1). Let $x$ be a generic Cohen real in the domain of $f$, and let $A=A_x$ and $B=B_{f(x)}$. By absoluteness of  $x\mapsto A_x$ and $y\mapsto B_y$, the set
$B$ can be defined in $V(A)$ from $A$ as follows: $B$ is the unique set satisfying that in any generic extension of $V(A)$, if $x'$ is such that $A=A_{x'}$, then $B=B_{f(x')}$. 
By assumption (2), $B$ is in $V$. By the forcing theorem, there is a condition $p$ in the Cohen poset (a non-meager set), forcing that $B_{f(\dot{x})}=\check{B}$. 
In particular, all generics $x$ extending $p$ are mapped to the same $F$-equivalence class.
Now the set of all $x\in p$ which are generic, over some large enough countable model, is a non-meager set sent to the  same $F$-class by $f$.
\end{proof}

Next we develop absolute classifications for (a comeager part of) the Bernoulli shift $P\curvearrowright \mathbb{R}^{\mathbb{N}}$.  By Lemma \ref{PermutationUltrahomogeneous}, it will suffice to consider
permutation groups $P$ of the form $\mathrm{Aut}(\mathcal{N})$, where $\mathcal{N}$ is some countable $\mathcal{L}$-structure.

\begin{definition}
Let $\mathcal{L}$ be a countable language. A {\bf countable $\mathcal{L}$-structure on $\mathbb{R}$} is any $\mathcal{L}$-structure $\mathcal{A}=(A,\ldots)$ whose domain $A$ is a countable subset of $\mathbb{R}$.
\end{definition}

\begin{example}\label{Ex:3}
Let $\mathcal{N}$ be a countable $\mathcal{L}$-structure and let the  $\mathrm{Aut}(\mathcal{N})\curvearrowright \mathrm{Inj}(\mathbb{N},\mathbb{R})$ be the injective part of the Bernoulli shift.  Consider the assignment:
\[x\mapsto\mathcal{A}_x\]
mapping every $x\in\mathrm{Inj}(\mathbb{N},\mathbb{R})$ to the countable $\mathcal{L}$-structure $\mathcal{A}_x=(A_x,\ldots)$ on $\mathbb{R}$, where $A_x=\{x_n\colon n\in \mathbb{N}\}$, and for all relations $R\in \mathcal{L}$ and functions  $f\in \mathcal{L}$ we have:
\begin{align*}
  \mathcal{A}_x\models R(x_{n_0},\ldots,x_{n_{k-1}})\quad  &\iff  \quad \mathcal{N}\models R(n_0,\ldots,n_{k-1})  \\
    \mathcal{A}_x\models f(x_{n_0},\ldots,x_{n_{k-1}})=x_n\quad &\iff\quad \mathcal{N}\models f(n_0,\ldots,n_{k-1})=n
\end{align*}
Then $x\mapsto\mathcal{A}_x$ is an absolute classification for $E(\mathrm{Aut}(\mathcal{N}))$ restricted to $\mathrm{Inj}(\mathbb{N},\mathbb{R})$.
This map is definable (set theoretically) and the fact that it is a complete classification is provable without any set theoretic assumptions. In particular, it defines an absolute classification in any generic extension. To see that it is an absolute classification it only remains to show that, given some $x\in\mathrm{Inj}(\mathbb{N},\mathbb{R})$, the invariant $\mathcal{A}_x$ is the same when calculated in some generic extension. This follows from the definition of $\mathcal{A}_x$.
\end{example}

We similarly have the following absolute classification for the labelled logic action from Section \ref{SS:LabelledAction}.

\begin{example}\label{Ex:4}
Let $\mathcal{L}$ be a countable language and let $\mathrm{Sym}(\mathbb{N})\curvearrowright \mathrm{Str}^{\mathrm{inj}}_{\mathbb{R}}(\mathcal{L},\mathbb{N})$ be the injective part of the labelled logic action. Consider the assignment:
\[(x,\mathcal{N})\mapsto\mathcal{A}_{x,\mathcal{N}}\]
which maps every $(x,\mathcal{N})\in\mathrm{Inj}(\mathbb{N},\mathbb{R})\times \mathrm{Str}(\mathcal{L},\mathbb{N})$ to the 
countable $\mathcal{L}$-structure $\mathcal{A}_{x,\mathcal{N}}$ on $\mathbb{R}$,  which is computed as $\mathcal{A}_{x}$ in Example \ref{Ex:3} with respect to $\mathcal{N}$. For every  $g\in\mathrm{Sym}(\mathbb{N})$ and all $(x,\mathcal{N})\in\mathrm{Inj}(\mathbb{N},\mathbb{R})\times \mathrm{Str}(\mathcal{L},\mathbb{N})$ and $R\in\mathcal{L}$, we have that:  
\begin{align*}
\mathcal{A}_{x,\mathcal{N}}\models R(x_{n_0},\ldots,x_{n_{k-1}}) \implies\\
\mathcal{N}\models R(n_0,\ldots,n_{k-1}) \implies\\
g\mathcal{N}\models R(g(n_0),\ldots,g(n_{k-1})) \implies \\
\mathcal{A}_{z,g\mathcal{N}}\models R(z_{g(n_0)},\ldots,z_{g(n_{k-1})}) \text{ for all $z\in \mathrm{Inj}(\mathbb{N})$},
\end{align*}
and since $(gx)_{g(n)}=x_n$, for all $n\in\mathbb{N}$, it follows that $\mathcal{A}_{gx,g\mathcal{N}}\models R(x_{n_0},\ldots,x_{n_{k-1}})$. Hence,  $(x,\mathcal{N})\mapsto\mathcal{A}_{x,\mathcal{N}}$ is a complete classification of $\simeq_{\mathrm{iso}}$ on $\mathrm{Str}^{\mathrm{inj}}_{\mathbb{R}}(\mathcal{L},\mathbb{N})$. As before, this map is an absolute classification.
 \end{example}

\subsection{Proofs of the main  results}

We may now formulate and prove our main theorem which is a common generalization of both Theorem \ref{T:1} and Theorem \ref{thm;BKL2-vs-F}. Notice that since the equivalence classes of $E(P)$ in $\mathbb{R}^{\mathbb{N}}$ are meager---see Section \ref{SS:BS}---the final claims in Theorem \ref{T:1} and Theorem \ref{thm;BKL2-vs-F} regarding the non-existence of Borel reductions, follow from the associated  strong ergodicity claims. For the definition of  algebraic dimension $\mathrm{dim}_{\mathcal{L}}(\mathcal{N})$ of an $\mathcal{L}$-structure $\mathcal{N}$, see Definition \ref{def:dim_for_structures}.

%Throughout this section, $\mathcal{M}$ is going to be a countable ultrahomogeneous $\mathcal{L}$-structures with the property that $[A]_{\mathrm{Aut}(\mathcal{M})}=\mathrm{acl}_{\mathcal{M}}(A)$, for every finite $A\subseteq \mathbb{N}$. In particular, terms like ``$\mathcal{M}$ has no algebraicity" in the sense of model theory will be synonymous to ``$\mathrm{Aut}(\mathcal{M})$ has no algebraicity" as a permutation group; see Section \ref{S:Prelim}. Recall that every Polish permutation group $P\leq \mathrm{Sym}(\mathbb{N})$ is of the form $\mathrm{Aut}(\mathcal{M})$, for some $\mathcal{M}$ as above; see Lemma \ref{PermutationUltrahomogeneous}.  

\begin{theorem}\label{thm: main theorem}
Let $n\in\omega$ and let $P\leq \mathrm{Sym}(\mathbb{N})$ be a locally finite, $(n+1)$-free Polish permutation group. Let also $\sigma$ be an $\mathcal{L}_{\omega_1,\omega}$-sentence so that $\mathrm{dim}_{\mathcal{L}}(\mathcal{N})\leq n$ for every $\mathcal{N}\in\mathrm{Str}(\sigma,\mathbb{N})$. Then, 
$E(P)$ is generically ergodic with respect to $\big( \mathrm{Str}^{\mathrm{inj}}_{\mathbb{R}}(\sigma,\mathbb{N}), \simeq_{\mathrm{iso}}\big)$.
\end{theorem}

Before we proceed to the proof of Theorem \ref{thm: main theorem}, we  recover from it Theorem \ref{T:1}.

\begin{proof}[Proof of Theorem \ref{T:1} from Theorem \ref{thm: main theorem}]
Let $P$ and $Q$ be as in the statement of  Theorem \ref{T:1}. By Lemma \ref{PermutationUltrahomogeneous} we have that $Q=\mathrm{Aut}(\mathcal{N})$ for some countable ultrahomogeneous 
$\mathcal{L}$-structure
with the property that $[A]_{Q}=\mathrm{acl}_{\mathcal{N}}(A)$, for every finite $A\subseteq \mathbb{N}.$
Let $\sigma$ be the Scott sentence of $\mathcal{N}$. Since every $\mathcal{N}'\in\mathrm{Str}(\sigma,\mathbb{N})$ is isomorphic to $\mathcal{N}$, we have that $\mathrm{Aut}(\mathcal{N}')$ is  of algebraic dimension less than or equal to $n$. 
By Theorem \ref{thm: main theorem} we conclude that $E(P)$ is generically ergodic with respect to $\big( \mathrm{Str}^{\mathrm{inj}}_{\mathbb{R}}(\sigma,\mathbb{N}), \simeq_{\mathrm{iso}}\big)$. Furthermore, by Section \ref{SS:LabelledAction} we have that $E_{\mathrm{inj}}( \mathrm{Aut}(\mathcal{N}))\leq_B \big( \mathrm{Str}^{\mathrm{inj}}_{\mathbb{R}}(\sigma,\mathbb{N}), \simeq_{\mathrm{iso}}\big),$ and so $E(P)$ is generically ergodic with respect to $E_{\mathrm{inj}}( \mathrm{Aut}(\mathcal{N}))=E(Q)$ as well. 
\end{proof}

We now turn to the proof of Theorem \ref{thm: main theorem}. We first fix some notation regarding Cohen forcing on $\mathbb{R}^{\mathbb{N}}$.
Consider the poset $\mathbb{P}$ whose elements are finite partial maps $p$ from $\mathbb{N}$  such that $p(n)\subseteq\mathbb{R}$ is an open interval with rational endpoints for every $n$ in the domain of $p$.
A condition $p$ {\bf extends} $q$, denoted by $p\leq q$, if $\mathrm{dom}(p)\supseteq\mathrm{dom}(q)$ and $p(n)\subseteq q(n)$ for every $n\in\mathrm{dom}(q)$. We fix some
 $\mathbb{P}$-generic $G\subseteq \mathbb{P}$ over $V$.
Working in $V[G]$, define $x=(x_n)$ to be the unique element of $\mathbb{R}^{\mathbb{N}}$ such that $x_n\in p(n)$, for every $p\in G$ with $n\in\mathrm{dom}(p)$.  We will refer to $x$ as the {\bf Cohen generic point of $\mathbb{R}^{\mathbb{N}}$}.
In what follows we also fix some  $\mathcal{L}$-structure $\mathcal{M}=(\mathbb{N},\ldots)$ in $V$ and let   \[\mathcal{A}:=\mathcal{A}_{x,\mathcal{M}}\] 
be the image of $(x,\mathcal{M})$ under the absolute classification defined in Example \ref{Ex:4}. We finally let $V(\mathcal{A})\subseteq V[G]$ be the symmetric model associated to $\mathcal{A}\in V[G]$.

We will need the following lemma which is interesting on its own right. The main point is that if $\mathcal{M}$ has enough symmetries, then the associated symmetric model $V(\mathcal{A})$ admits an analysis of the definable supports  similar to the one in \cite{HL64}, for  the basic Cohen model. 

Recall the definition of definable supports in $V(\mathcal{A})$, from Fact~\ref{Fact:1}. Note that any set in the transitive closure of $\mathcal{A}$ is definable (in $V(\mathcal{A})$) using $\mathcal{A}$ and finitely many reals in $A$. Hence, can replace any definable support by a finite subset of $A$. From now on we assume that a definable support $\bar{a}$ in $V(\mathcal{A})$ enumerates a subset of $A$.

\begin{lemma}\label{lem;supports-for-V(A)}
Assume that  $\mathrm{Aut}(\mathcal{M})$ is locally finite. Take any set $S\in V(\mathcal{A})$ such that $S\subseteq V$, and let $\bar{a}$ be a definable support for $S$ in $V(\mathcal{A})$. 
If $[\bar{a}]\subseteq A$ is the algebraic closure of $\bar{a}$  with respect to $\mathrm{Aut}(\mathcal{M})$, then $S\in V([\bar{a}])$.
\end{lemma}
\begin{proof}
By local finiteness we may assume without loss of generality, by enlarging $\bar{a}$, that $\bar{a}=\left(a_0,\dots,a_{k-1}\right)$ enumerates $[\bar{a}]$.
Fix a formula $\phi$ and fix some parameter $v\in V$ such that $S=\big\{s\in V\mid V(\mathcal{A})\models\varphi(s,v,\bar{a}, \mathcal{A})\big\}$. 
By reflection, let $\xi$ be a large enough ordinal so that $S\subseteq V_\xi$ and so that for any $p\in\mathbb{P}$ we have:
\[p\Vdash \varphi^{V(\mathcal{\dot{A}})}(\check{s},\check{v},\dot{\bar{a}},\dot{\mathcal{A}}) \text{ if and only if } V_\xi\models \big( p\Vdash \varphi^{V(\dot{\mathcal{A}})}(\check{s},\check{v},\dot{\bar{a}},\dot{\mathcal{A}})\big).\]

Fix  $n_0,\ldots,n_{k-1}\in\mathbb{N}$ so that $a_i = x_{n_i}$.
Working in $V(\bar{a})$ we define
\[S'=\{s\in V_\xi \mid \text{ for any } p, \text{ if } a_i\in p(n_i) \text{ for } i<k, \text{ then } V_\xi\models (p\Vdash \varphi^{V(\mathcal{\dot{A}})}(\check{s},\check{v},\dot{\bar{a}}, \dot{\mathcal{A}}))\}.\]
Here by $\dot{\bar{a}}$ we mean $\dot{x}_{n_0},\dots,\dot{x}_{n_{k-1}}$.
Since the forcing relation is definable, $S'$ is in $V(\bar{a})$. 
We will conclude the proof of the lemma by showing that $S=S'$.

To  see that $S'\subseteq S$, let $s\in S'$ and take any condition $p\in G$ such that $n_0,...,n_{k-1}\in\mathrm{dom}(p)$. By definition of $S'$ and by the choice of $\xi$, $p\Vdash\varphi^{V(\mathcal{\dot{A}})}(\check{s},\check{v},\dot{\bar{a}}, \dot{\mathcal{A}})$. Since $p\in G$ we have that $\varphi^{V(\mathcal{A})}(s,v,\bar{a}, \mathcal{A})$ holds in $V[G]$, and therefore, $\varphi(s,v,\bar{a}, \mathcal{A})$ holds in $V(\mathcal{A})$, and so $s\in S$.
The converse direction  $S\subseteq S'$  follows from the next claim.
\begin{claim}
For any $p,q\in\mathbb{P}$, if $a_i\in p(n_i)\cap q(n_i)$ for $i<k$, then for any $s\in V$,
\[p\Vdash \varphi^{V(\dot{\mathcal{A}})}(\check{s},\check{v},\dot{\bar{a}}, \dot{\mathcal{A}})\iff q\Vdash \varphi^{V(\dot{\mathcal{A}})}(\check{s},\check{v},\dot{\bar{a}}, \dot{\mathcal{A}}).\]
\end{claim}
\begin{proof}
 Assume towards a contradiction that $p,q$ are as in the claim but we additionally have that $p\Vdash \varphi^{V(\dot{\mathcal{A}})}(\check{x},\check{v},\dot{\bar{a}}, \dot{\mathcal{A}})$ and $q\Vdash \neg\varphi^{V(\dot{\mathcal{A}})}(\check{x},\check{v},\dot{\bar{a}}, \dot{\mathcal{A}})$.
Without loss of generality we may also assume that $p\in G$.
Let $L:=\{m_0,\ldots,m_{\ell-1}\}\subseteq \mathbb{N}=\mathrm{dom}(\mathcal{M})$ be disjoint from $K:=\{n_0,...,n_{k-1}\}$, so that $L\cup K$     contains the domains of $p$ and $q$. 
Since $\bar{a}=[\bar{a}]_{\mathrm{Aut}(\mathcal{M})}$, we have that $K=[K]_{\mathrm{Aut}(\mathcal{M})}$. Hence, for every $\ell\in \mathbb{N}\setminus K$ the orbit of $\ell$ under the stabilizer 
$\mathrm{Aut}(\mathcal{M})_K$ of $K$ is infinite. By the Neumann lemma \cite{neumann1976structure}*{Lemma 2.3}, for any finite $\Delta\subset \mathbb{N}\setminus K$ there is some $\pi\in \mathrm{Aut}(\mathcal{M})_K$ so that $\Delta \cap \pi\Delta=\emptyset$. By repeated applications of the Neumann lemma, we may find, in the ground model $V$, an infinite sequence $(\pi_j)_j$  in $\mathrm{Aut}(\mathcal{M})$ so that:
\begin{enumerate}
    \item $\pi_j$ fixes $K$ pointwise; and
    \item $\pi_i(L)\cap \pi_j(L) = \emptyset $ for all $i\neq j$.
\end{enumerate}

By genericity of $G$ there is some  $j$ such that $x\circ\pi_j$ extends  $q$.
Set $\pi:=\pi_j$ and let  $G':=\{p\circ\pi \colon p\in G\}$ and $x':=x\circ \pi$. We have that: 
\begin{enumerate}
    \item[(i)] $G'$ is $\mathbb{P}$-generic and $V[G']=V[G]$, since $\pi\in V$;
    \item[(ii)] $x'(n_i)=x(n_i)=a_i$ and therefore $\dot{\bar{a}}[G']=\bar{a}$, by (1) above;
    \item[(iii)] $\dot{\mathcal{A}}[G']=\dot{\mathcal{A}}[G]=\mathcal{A}$, since $x\mathrel{E(\mathrm{Aut}(\mathcal{M}))}x'$; see Example~\ref{Ex:3}.
\end{enumerate}
By (ii),(iii), and since $x'$ extends $q$, we have that $\neg\varphi^{V(\mathcal{A})}(s,v,\bar{a},\mathcal{A})$ holds in $V[G']$. 
On the other hand, since $\dot{\mathcal{A}}[G]=\mathcal{A}$ and $x$ extends $p$, we have that $\varphi^{V(\mathcal{A})}(s,v,\bar{a},\mathcal{A})$ holds in $V[G]$.
% I commented out the line you added here, attempting to get a more direct contradiction as below:
%By fact [${V(\mathcal{A})}^{V[g]}={V(\mathcal{A})}^{V[g']}$] we have that $\neg\varphi(x,v,\bar{a},\mathcal{A})$ holds in $V(\mathcal{A})$ which again by fact  [${V(\mathcal{A})}^{V[g]}={V(\mathcal{A})}^{V[g']}$] contradicts that $\varphi^{V(\mathcal{A})}(x,v,\bar{a},\mathcal{A})$, as computed in $V[g]$.
We conclude that $V(\mathcal{A})$ satisfies both $\neg\varphi(s,v,\bar{a},\mathcal{A})$ and $\varphi(s,v,\bar{a},\mathcal{A})$, a contradiction.
\end{proof}
This concludes the proof of Lemma \ref{lem;supports-for-V(A)}
\end{proof}

We will use the following folklore ``mutual genericity fact''.
\begin{fact}
Suppose $F_1,F_2\subseteq A$ are finite. Then $V(F_1)\cap V(F_2)=V(F_1\cap F_2)$.
\end{fact}
The reason is as follows. We may write $F_1=F\cup a_1$ and $F_2=F\cup a_2$ where $F= F_1\cap F_2$ and $a_1,a_2$ are disjoint finite subsets of $A$. Now $V(F_i)=V(F)[a_i]$, where $a_i$ is generic over $V(F)$ for the poset $\mathbb{Q}_i$ for adding a real in $\mathbb{R}^{|a_i|}$. Furthermore, $(a_1,a_2)$ can be seen as a generic for the poset $\mathbb{Q}_1\times\mathbb{Q}_2$.
It follows that $V(F)[a_1]\cap V(F)[a_2]=V(F)$.

Let $S\in V(\mathcal{A})$ with $S\subseteq V$ and let $F\subseteq A =\mathrm{dom}(\mathcal{A})$ be a set. We say that $F$ is \textbf{a support for $S$} if $S\in V(F)$.
We define \textbf{the support $\mathrm{supp}(S)$  of $S$} to be the intersection of all  $F\subseteq A$ which are supports for $S$. Under the assumptions of the previous lemma, for every $S$ as above,  $\mathrm{supp}(S)$ exists and is a finite subset of $A$ such that  $S\in V(\mathrm{supp}(S))$.

We may now conclude with the proof of Theorem \ref{thm: main theorem}.

\begin{proof}[Proof of Theorem \ref{thm: main theorem}]
By Lemma~\ref{PermutationUltrahomogeneous} we may write $P=\mathrm{Aut}(\mathcal{M})$ for a ultrahomogeneous structure $\mathcal{M}$ so that the closure operations $[-]_P$ and $\mathrm{acl}_{\mathcal{M}}(-)$ coincide.

We will apply  Lemma~\ref{lem;generic-erg-symm-model}~$(2)\implies(1)$. For that, let $x\in\mathrm{Inj}(\mathbb{N},\mathbb{R})$ be the Cohen generic point associated to the $\mathbb{P}$-generic filter $G$,  and set $\mathcal{A}:=\mathcal{A}_{x}$ the corresponding invariant for the shift action $P\curvearrowright\mathrm{Inj}(\mathbb{N},\mathbb{R})$ as in Example~\ref{Ex:3}. Let $\mathcal{B}=(B,...)$ be a potential invariant for $\big( \mathrm{Str}^{\mathrm{inj}}_{\mathbb{R}}(\sigma,\mathbb{N}), \simeq_{\mathrm{iso}}\big)$ in $V(\mathcal{A})$ which is definable only from $\mathcal{A}$ and parameters from $V$. Recall that by potential invariant we mean that there are  $y\in\mathrm{Inj}(\mathbb{N},\mathbb{R})$ and  $\mathcal{N}\in \mathrm{Str}(\sigma,\mathbb{N})$, in some further generic extension of $V(\mathcal{A})$, so that $\mathcal{B}:=\mathcal{B}_{y,\mathcal{N}}$.
To apply Lemma~\ref{lem;generic-erg-symm-model}, it remains to show that $\mathcal{B}\in V$.

\begin{claim}
Set $B:=\mathrm{dom}(\mathcal{B})$. We have that  $B\subseteq V$.
\end{claim}
\begin{proof}[Proof of Claim]

Assume towards a contradiction that there is some $b\in B$ not in $V$.
Let $F$ be the support of $b$. This is well defined by Lemma~\ref{lem;supports-for-V(A)}, as $b$ is a real, and therefore can be thought of as a subset of $V$.
Since $b\in V(F)$, it follows that F is not empty. Let  $m_0,\ldots,m_{k-1}\in\mathbb{N}$ with $F=\{x_{m_i}\colon i<k\}$.
Fix a condition $p$ forcing that $\dot{b}\in \dot{B}$ and that $\{\dot{x}_{m_i}: i<k\}$ is the support of $\dot{b}$. Here dotted symbols are some fixed names for the corresponding objects in the generic extension. In particular $\dot{x}$ is the name for the Cohen generic sequence of reals.
By local finiteness of $\mathcal{M}$ we may choose some finite  $K\subseteq \mathbb{N}=\mathrm{dom}(\mathcal{M})$ which contains  \[\mathrm{dom}(p)\cup\{m_0,\ldots,m_{k-1}\}\] 
and it is algebraically closed, i.e., $[K]_{\mathrm{Aut}(\mathcal{M})}=K$. Since $\mathcal{M}$ is $(n+1)$-free, there are $\pi_0,\ldots,\pi_n\in \mathrm{Aut}(\mathcal{M})$ so that for all $i\leq n$ we have  that
\[[\pi_i(K)]_{\mathrm{Aut}(\mathcal{M})} \; \text{ is disjoint from } \; [\bigcup_{j: j\neq i} \pi_j(K)]_{\mathrm{Aut}(\mathcal{M})} \]
Notice now that  $[\emptyset]_{\mathrm{Aut}(\mathcal{M})}=\emptyset$. Indeed, any $(n+1)$-free permutation group is in particular $1$-free; and any locally finite $1$-free permutation group $P$ satisfies $[\emptyset]_P=\emptyset$.
Hence, the orbit of each $\ell\in [\pi_i(K)]_{\mathrm{Aut}(\mathcal{M})}$ is infinite for  all $i\leq n$. By the Neumann lemma \cite{neumann1976structure}*{Lemma 2.3}, for every finite $E\subseteq \mathbb{N}$ there is some  $\pi_E\in \mathrm{Aut}(\mathcal{M})$, so that    $(\pi_E\circ \pi_i)(K)\cap E=\emptyset$ for all $i\leq n$.
Hence, by the genericity of $x\in\mathrm{Inj}(\mathbb{N},\mathbb{R})$ we may assume that for all $i\leq n$ we have that $x^i:=x\circ\pi_i$ extends $p$. Let $G_i:=\{q\circ\pi_i\colon q\in G \}$ be the associated filters. Since all $\pi_i$ above can be chosen in $V$, we have that each $G_i$ is generic and $V[G]=V[G_i]$. 

For each $i\leq n$, working in $V[G_i]$, since $p\in G_i$ and $x^i\mathrel{E(P)}x$, we have that: the realization of $\dot{\mathcal{A}}$ is the same $\mathcal{A}$; the realization of $\dot{\mathcal{B}}$ is the same $\mathcal{B}$; the interpretation of $\dot{F}$ is some $F_i\subseteq A$ such that $\set{F_i}{i=0,...,n}$ are pairwise disjoint; and $F_i$ is the support of the interpretation $b_i$ of $\dot{b}$  in  $V[G_i]$.
Since $p\in G_i$, we have that $b_i\in B$ for all $i\leq n$.
%Since $V[G_i]=V[G]$ we have that $b_i\in B$ for all $i\leq n$.
By assumption, there is some $i$ such that, in $\mathcal{B}$, $b_i$ is in the $\mathcal{L}$-algebraic closure $\mathrm{acl}_{\mathcal{B}}(\set{b_j}{j\neq i})$. 
Here we are using that  ``for all $\mathcal{N}\in \mathrm{Str}(\sigma,\mathbb{N})$ and every $b_0,\ldots,b_n\in\mathrm{dom}(\mathcal{N})$, there is some $i\leq n$ so that $b_i\in\mathrm{acl}_{\mathcal{N}}(\{b_j\colon j\neq i\})$" is equivalent to a $\mathbf{\Pi}^1_1$ statement, and therefore absolute. Specifically, the set of all $\mathcal{N}$ satisfying the desired property is a Borel subset of $\mathrm{Str}(\sigma,\mathbb{N})$.

It follows that $b_i$ is definable in $V(\mathcal{A})$ from $\set{b_j}{j\neq i}$ and $\mathcal{B}$. This is because there is a finite set of reals, definable from $\set{b_j}{j\neq i}$ and $\mathcal{B}$, which contains $b_i$, so $b_i$ can be defined using a definable linear ordering of the reals.
Recall that $\mathcal{B}$ is definable from $\mathcal{A}$, by assumption. Since $\seqq{b_j}{j\neq i}$ is definable from $\seqq{F_j}{j\neq i}$, it follows that $\bigcup_{j\colon j\neq i} F_j$ is a definable support for $b_i$ in $V(\mathcal{A})$ and by Lemma~\ref{lem;supports-for-V(A)} we have that $b_i\in V([\bigcup_{j\colon j\neq i} F_j]_{\mathrm{Aut}(\mathcal{A})})$. By mutual genericity we have
\[  V(F_i) \bigcap V([\bigcup_{j\colon j\neq i} F_j]_{\mathrm{Aut}(\mathcal{A})})  =  V(F_i \; \bigcap \; [\bigcup_{j\colon j\neq i} F_j]_{\mathrm{Aut}(\mathcal{A})})=V  \]
But then, $\mathrm{supp}(b_i)=\emptyset$, contradicting the assumption that   $\mathrm{supp}(b)\neq \emptyset$.
\end{proof}
By assumption, $\emptyset$ is a definable support for $B$, and so $B\in V$, by Lemma~\ref{lem;supports-for-V(A)}.
To conclude that $\mathcal{B}\in V$, we need to show that for any symbol $R\in \mathcal{L}$, its interpretation $R^{\mathcal{B}}$ is in $V$.
By the claim, $R^{\mathcal{B}}$ is a subset of $V$. By assumption, $\emptyset$ is a definable support for $R^{\mathcal{B}}$ (as the entire structure $\mathcal{B}$ is definable from $\mathcal{A}$ without parameters). It follows from Lemma~\ref{lem;supports-for-V(A)} that $R^{\mathcal{B}}\in V$, as required.
\end{proof}
Recall from Example~\ref{Ex:2} that for $n\geq 2$, $\mathrm{Aut}(\mathcal{N}^{\mathrm{BKL}}_n)$ is locally finite, $n$-free, and has algebraic dimension precisely $n$.
It follows from Theorem~\ref{T:1} that
\begin{equation*}
    E_{\mathrm{inj}}(\mathrm{Aut}(\mathcal{N}^{\mathrm{BKL}}_{n+1}))\textrm{ is generically }E_{\mathrm{inj}}(\mathrm{Aut}(\mathcal{N}^{\mathrm{BKL}}_n))\textrm{-ergodic}.
\end{equation*}
This concludes Corollary~\ref{cor:1}.
In particular, $E_{\mathrm{inj}}(\mathrm{Aut}(\mathcal{N}^{\mathrm{BKL}}_{n+1}))$ is not Borel reducible to $E_{\mathrm{inj}}(\mathrm{Aut}(\mathcal{N}^{\mathrm{BKL}}_n))$.
In the other direction, we do not know if $E_{\mathrm{inj}}(\mathrm{Aut}(\mathcal{N}^{\mathrm{BKL}}_n))$ is Borel reducible to $E_{\mathrm{inj}}(\mathrm{Aut}(\mathcal{N}^{\mathrm{BKL}}_{n+1}))$.

\subsection{A strictly $\leq_B$-increasing sequence}\label{subsec : positive reduction}

Recall that each structure $\mathcal{N}_n^{\mathrm{BKL}}$ is defined as the unique ultrahomogenous $\mathcal{L}_n$-structure satisfying properties (1)-(4) in Example~\ref{Ex:2}. These properties can be captured by an $\mathcal{L}_{\omega_1,\omega}$-sentence $\sigma_n$. The equivalence relations $(\mathrm{Str}_{\mathbb{R}}^{\mathrm{inj}}(\sigma_n,\mathbb{N}),\simeq_{\mathrm{iso}})$ were studied in \cites{KP}, where it was shown that they are pairwise incompatible with respect to $*$-reductions\footnote{In \cites{KP}, structures satisfying $\sigma_n$ were called  $\mathrm{BKL}_n$-structures and $\mathrm{Str}_{\mathbb{R}}^{\mathrm{inj}}(\sigma_n,\mathbb{N})$ was denoted by $\mathrm{Mod}_\omega(\widehat{\mathcal{B}_n})$}. Their relationship with respect to Borel reducibility was left open.
We show that, with respect to Borel reducibility, these form a strictly increasing sequence.

\begin{proposition}
 $(\mathrm{Str}_{\mathbb{R}}^{\mathrm{inj}}(\sigma_n,\mathbb{N}),\simeq_{\mathrm{iso}})<_B(\mathrm{Str}_{\mathbb{R}}^{\mathrm{inj}}(\sigma_{n+1},\mathbb{N}),\simeq_{\mathrm{iso}})$.
\end{proposition}
\begin{proof}
Since $E_{\mathrm{inj}}(\mathrm{Aut}(\mathcal{N}^{\mathrm{BKL}}_{n+1}))\leq_B(\mathrm{Str}_{\mathbb{R}}^{\mathrm{inj}}(\sigma_{n+1},\mathbb{N}),\simeq_{\mathrm{iso}})$, it follows from Theorem~\ref{T:1} that $(\mathrm{Str}_{\mathbb{R}}^{\mathrm{inj}}(\sigma_{n+1},\mathbb{N}),\simeq_{\mathrm{iso}})$ is not Borel reducible to $(\mathrm{Str}_{\mathbb{R}}^{\mathrm{inj}}(\sigma_n,\mathbb{N}),\simeq_{\mathrm{iso}})$.

We now turn to the positive reduction. The point is that there is a simple way to code an $\mathcal{L}_n$-structure which satisfies $\sigma_n$ as a $\mathcal{L}_{n+1}$-structure which satisfies $\sigma_{n+1}$. This is done as follows.

Let $\mathcal{A}=(A,R_j,s_j)_{j\in\omega}$ be an $\mathcal{L}_n$ structure. Define a structure $\mathcal{\bar{A}}=(A,\bar{R}_j,\bar{s}_j)_{j\in\omega}$ as follows, where $\bar{R}_j$ are ${n+1}$-ary relations and $\bar{s}_j$ are $n+1$-ary function symbols.
\begin{itemize}
    \item $\bar{R}_j(a_0,...,a_n)\iff R_j(a_0,...,a_{n-1})$;
    \item $\bar{s}_j(a_0,...,a_n)=s_j(a_0,...,a_{n-1})$.
\end{itemize}
If $\mathcal{A}$ satisfies $\sigma_n$, then $\mathcal{\bar{A}}$ satisfies $\sigma_{n+1}$. Moreover, given $\mathcal{A}_1,\mathcal{A}_2$, then $\mathcal{A}_1\simeq \mathcal{A}_2\iff \mathcal{\bar{A}}_1\simeq \mathcal{\bar{A}}_2$. Moreover, the map from $\mathrm{Str}(\mathcal{L}_n,\mathbb{N})$ to $\mathrm{Str}(\mathcal{L}_{n+1},\mathbb{N})$ sending $\mathcal{A}=(\mathbb{N},\bar{R}_j,\bar{s}_j)_{j\in\omega}$ to $\mathcal{\bar{A}}=(\mathbb{N},\bar{R}_j,\bar{s}_j)_{j\in\omega}$ is Borel. This induces a Borel map from $\mathrm{Str}_{\mathbb{R}}^{\mathrm{inj}}(\sigma_n,\mathbb{N})$ to $\mathrm{Str}_{\mathbb{R}}^{\mathrm{inj}}(\sigma_n,\mathbb{N})$ which is the desired reduction.
\end{proof}

\section{Unpinned equivalence relations}\label{sec;pinned}

We recall the definition of pinned equivalence relations and pinned cardinals and explain the results in \cite{Zap11} about the equivalence relation $=^+\restriction Z$. We then present the equivalence relations $E_{\mathrm{inj}}(\mathrm{Aut}(\mathcal{N}^{\mathrm{BKL}}_{n+1}))$ in this context and give another proof that $E_{\mathrm{inj}}(\mathrm{Aut}(\mathcal{N}^{\mathrm{BKL}}_{n+1})\not\leq_B E_{\mathrm{inj}}(\mathrm{Aut}(\mathcal{N}^{\mathrm{BKL}}_{n})$, a weak form of Corollary~\ref{cor:1}.
Finally, we prove Theorem~\ref{thm;BKL2-vs-F}.

\begin{definition}[\cite{LZ}]\label{def;pinned}
Let $E$ be an analytic equivalence relation on a Polish space $X$. Let $\mathbb{P}$ be a poset and $\tau$ a $\mathbb{P}$-name. 
\begin{itemize}
    \item The name $\tau$ is $E$-pinned if $\mathbb{P}\times \mathbb{P}$ forces that $\tau_{l}$ is $E$-equivalent to $\tau_{r}$, where $\tau_{l}$ and $\tau_{r}$ are the interpretation of $\tau$ using the left and right generics respectively.
    \item If $\tau$ is $E$-pinned the pair $\left<\mathbb{P},\tau\right>$ is called an $E$-\textbf{pin}.
    \item An $E$-pin $\left<\mathbb{P},\sigma\right>$ is \textbf{trivial} if there is some $x\in X$ such that $\mathbb{P}\Vdash \sigma\mathrel{E}\check{x}$.
    \item $E$ is \textbf{pinned} if all $E$-pins are trivial.
    \item Given two $E$-pins $\left<\mathbb{P},\sigma\right>$ and $\left<\mathbb{Q},\tau\right>$, say that they are $\bar{E}$-equivalent, $\left<\mathbb{P},\sigma\right>\mathrel{\bar{E}}\left<\mathbb{Q},\tau\right>$, if $\mathbb{P}\times \mathbb{Q} \Vdash \sigma \mathrel{E}\tau$.
    \item The \textbf{pinned cardinal of $E$}, $\kappa(E)$, is the smallest $\kappa$ such that every $E$-pin is $\bar{E}$-equivalent to an $E$-pin with a post of size $<\kappa$.
\end{itemize}
\end{definition}
\begin{lemma}[\cite{LZ}]\label{lem;pin-card}
If $E$ is Borel reducible to $F$ then $\kappa(E)\leq \kappa(F)$.
\end{lemma}

For equivalence relations which are classifiable by countable structures, the pinned cardinal can be calculated more easily by the size of ``potential invariants'', as explained below.
\begin{remark}
Ulrich, Rast, and Laskowski \cite{URL17} have independently developed the notion of pinned cardinality in the special case of isomorphism relations.
The presentation below is essentially equivalent to the one in \cite{URL17}.
\end{remark}

Assume that $E$ is a Borel equivalence relation which admits an absolute classifiable $x\mapsto A_x$.
Say that a set $A$ is a \textbf{potential E-invariant} if in some forcing extension there is an $x$ in the domain of $E$ such that $A=A_x$.
If $A$ is a potential invariant for $E$, say that $A$ is \textbf{trivial} if there is an $x$  in the ground model such that $A=A_x$.

\begin{proposition}
There is a one-to-one correspondence between
\begin{itemize}
    \item $E$ pins $\left< \mathbb{P},\tau\right>$, and
    \item potential invariants $A$,
\end{itemize}
such that $\left<\mathbb{P},\tau\right>$ is trivial if and only if $A$ is trivial.
Specifically, a potential invariant $A$ corresponds to $\left<\mathbb{P},\tau\right>$ if and only if $\mathbb{P}\Vdash A_{\tau}=\check{A}$.
\end{proposition}

Assume first that $\left<\mathbb{P},\tau\right>$ is an $E$-pin. Let $G_l\times G_r$ be $\mathbb{P}\times \mathbb{P}$-generic and let $x_l,x_r$ be the interpretations of $\tau$ according to $G_l$,$G_r$ respectively.
Since $x_l$ and $x_r$ are $E$-related, it follows that $A=A_{x_l}=A_{x_r}$.
Furthermore, this set $A$ is in the intersection $V[G_l]\cap V[G_r]$, which is equal to $V$ by mutual genericity.
If $\left<\mathbb{P},\tau\right>$ is trivial there is $x\in V$ such that $x\mathrel{E}x_l$. In particular, $A_x=A_{x_l}=A$.
Conversely, if there is $x\in V$ with $A=A_x$ then $x$ witnesses that $\left<\mathbb{P},\tau\right>$ is trivial: given any $\mathbb{P}$-generic $G$ over $V$, $A_{\tau[G]}=A=A_x$, so $\tau[G]$ is $E$-related to $x$.

Now let $A$ be a potential invariant for $E$. By assumption there is a poset $\mathbb{Q}$, a generic $G$ and $x\in V[G]$ such that $A=A_x$.
Let $\tau$ be a $\mathbb{Q}$-name such that $\tau[G]=x$. 
Fix a condition $q\in\mathbb{Q}$ such that $q$ forces that $A_{\tau}=A$ and define $\mathbb{P}=\mathbb{Q}\restriction p$.
Now $\left<\mathbb{P},\tau\right>$ is an $E$-pin such that $\mathbb{P}\Vdash A_\tau=A$.

%Given a potential invariant $A$ and some model containing $A$, there will be an $x$ such that $A=A_x$ if and only if there is a countable enumeration of the transitive closure of $A$. So for any poset $\mathbb{Q}$, there is a $\mathbb{Q}$-name $\tau$ such that $\mathbb{Q}\Vdash A_{\tau}=A$ if and only if $\mathbb{Q}$ collapses the cardinality of the transitive closure of $A$ to be countable. In particular, there is a $\sigma$ such that $\left<\mathbb{Q},\sigma\right>$ is $\bar{E}$-related to $\left<\mathbb{P},\tau\right>$ if and only if $\mathbb{Q}$ collapses the cardinality of the transitive closure of $A$ to be countable.

\begin{corollary}
The pinned cardinal of $E$ is $\kappa$ if and only if any potential $E$-invariant is trivial in a generic extension by a poset of size $<\kappa$.
\end{corollary}

\begin{example}
Consider the equivalence relation $=^+$ on $\mathbb{R}^\mathbb{N}$ with the complete classification $\mathbb{R}^\mathbb{N}\ni x\mapsto \set{x(n)}{n\in\mathbb{N}}$.
The potential invariants of $=^+$ are precisely all sets of reals. 
%This is because any set of reals $A$ is equal to $A_x$ where $x$ enumerates $A$ in a generic extension collapsing $|A|$.
Therefore the pinned cardinal of $=^+$ is $\mathfrak{c}^+$.
\end{example}

To find an unpinned equivalence relation strictly below $=^+$, Zapletal \cites{Zap11} restricted $=^+$ to an invariant subset, in the following way, precisely to limit the size of its potential invariants.

\begin{example}[\cites{Zap11}]\label{example: Zapletal}
Fix Borel functions $f_n\colon\mathbb{R}\to\mathbb{R}$, $n\in\mathbb{N}$, such that the graph $\mathcal{G}$ on $\mathbb{R}$, defined by 
\begin{equation*}
    x\mathrel{\mathcal{G}}y\iff\exists n(f_n(x)=y\vee f_n(y)=x),
\end{equation*}
has cliques of size $\aleph_1$, but no cliques of size greater than $\aleph_1$ (see \cite{Zap11}*{Fact 2.2}).
Let $Z\subseteq\mathbb{R}^\omega$ be the $G_\delta$ set of all injective countable sequences of reals which enumerate a clique in $\mathcal{G}$, and consider the equivalence relation $=^+\restriction Z$ on $Z$.
The potential invariants of $=^+\restriction Z$ are precisely all sets of reals which form a clique in $\mathcal{G}$, and those have cardinality $\leq\aleph_1$. Therefore the pinned cardinal of $=^+\restriction Z$ is $\aleph_1^+$.
\end{example}

\begin{corollary}[\cite{Zap11}]
$=^+$ is not Borel reducible to $=^+\restriction Z$.
\end{corollary}
\begin{proof}
By the absoluteness of Borel reducibility, we may work in some forcing extension where the continuum hypothesis fails, that is, $\mathfrak{c}>\aleph_1$.
In this model the pinned cardinal of $=^+$ is strictly greater than that of $=^+\restriction Z$, so the corollary follows from Lemma~\ref{lem;pin-card}.
\end{proof}

\begin{example}
By Example~\ref{Ex:3},
a potential invariant for $E_{\mathrm{inj}}(\mathrm{Aut}(\mathcal{N}^{\mathrm{BKL}}_{n+1}))$ is a set of reals together with a $\mathrm{BKL}_{n+1}$-structure on it. Since $\mathrm{BKL}_{n+1}$ models are of size at most $\aleph_n$, and there is a model of size $\aleph_n$, it follows that the pinned cardinal of $E_{\mathrm{inj}}(\mathrm{Aut}(\mathcal{N}^{\mathrm{BKL}}_{n+1}))$ is $\aleph_{n}^+$, in a model where $|\mathbb{R}|\geq\aleph_n$.
It follows from Lemma~\ref{lem;pin-card} that $E_{\mathrm{inj}}(\mathrm{Aut}(\mathcal{N}^{\mathrm{BKL}}_{n+1}))$ is not Borel reducible to $E_{\mathrm{inj}}(\mathrm{Aut}(\mathcal{N}^{\mathrm{BKL}}_{n}))$.
\end{example}
Both $E_{\mathrm{inj}}(\mathrm{Aut}(\mathcal{N}^{\mathrm{BKL}}_{2}))$ and $=^+\restriction Z$ have the same pinned cardinal $\aleph_1^+$. Using Theorem~\ref{thm: main theorem} we are able to separate them.

\begin{proof}[Proof of Theorem \ref{thm;BKL2-vs-F}]
Consider a language with countably many function symbols $h_i$ and the theory $\sigma$ asserting that for any $x$ and $y$ there is some $i$ such that either $h_i(x)=y$ or $h_i(y)=x$. Notice that every structure in $\mathrm{Str}(\sigma,\mathbb{N})$ has algebraic dimension less than or equal to $1$. By Theorem \ref{thm: main theorem} and the observation in   Example \ref{Ex:2} that  $\mathrm{Aut}(\mathcal{N}_2^{\mathrm{BKL}})$ is $2$-free we have that  $E\big( \mathrm{Aut}(\mathcal{N}^{\mathrm{BKL}}_2)\big)$ is generically ergodic with respect to   $\big(\mathrm{Str}^{\mathrm{inj}}_{\mathbb{R}}(\sigma,\mathbb{N}),
\simeq_{\mathrm{iso}} \big)$.
Therefore it suffices to show that $=^+\restriction Z$ is Borel reducible to $\big(\mathrm{Str}^{\mathrm{inj}}_{\mathbb{R}}(\sigma,\mathbb{N}),
\simeq_{\mathrm{iso}} \big)$.
This is done by sending $x\in Z$ to the pair $(x,\mathcal{N})$ where $\mathcal{N}$ is the structure on $\mathbb{N}$ defined by $h_i(n)=m\iff f_i(x(n))=x(m)$ (where $f_i$ are the functions from Example~\ref{example: Zapletal}).
\end{proof}

\end{document}